\title{Geodesics of norms on the contactomorphisms group of $\mathbb{R}^{2n}\times S^1$ }
\author{Pierre-Alexandre Arlove}
\date{\vspace{-7.ex}}
\documentclass[a4paper]{article}
\usepackage[T1]{fontenc}
\usepackage[english]{babel}
\usepackage{lmodern}
\usepackage{amssymb,amsmath,amsfonts,amsxtra,amsthm}
\usepackage{mathrsfs}
\usepackage{hyperref}
\usepackage{setspace}
\usepackage{multicol}
\usepackage{mathtools}
\usepackage{mathenv}
\usepackage{color}
\usepackage{babelbib}
\usepackage{enumitem}
\usepackage{graphicx}
\usepackage{mathabx}
\usepackage[all]{xy}
\DeclareMathOperator{\Interior}{Interior}
\DeclareMathOperator{\Supp}{Supp}
\DeclareMathOperator{\Id}{Id}
\DeclareMathOperator{\Ker}{Ker}

\DeclareMathOperator{\Cont}{Cont}
\DeclareMathOperator{\Hess}{Hess}
\theoremstyle{definition}
\newtheorem{definition}{Definition}[section]

\theoremstyle{plain}

\newtheorem{theoreme}[definition]{Theorem}
\newtheorem{proposition}[definition]{Proposition}

\newtheorem{lemme}[definition]{Lemma}
\newtheorem{corollaire}[definition]{Corollary}
\theoremstyle{remark}

\newtheorem{remarque}[definition]{Remark}

\usepackage[11pt]{extsizes}
\makeatletter
\newcounter {subsubsubsection}[subsubsection]
\renewcommand\thesubsubsubsection{\thesubsubsection .\@alph\c@subsubsubsection}
\newcommand\subsubsubsection{\@startsection{subsubsubsection}{4}{\z@}%
	{-3.25ex\@plus -1ex \@minus -.2ex}%
	{1.5ex \@plus .2ex}%
	{\normalfont\normalsize\bfseries}}
\newcommand*\l@subsubsubsection{\@dottedtocline{3}{10.0em}{4.1em}}
\newcommand*{\subsubsubsectionmark}[1]{}
\makeatother

\begin{document}
\maketitle

\setlength{\parindent}{0pt}


\begin{abstract}
   We prove that some paths of contactomorphisms of $\mathbb{R}^{2n}\times S^1$ endowed with its standard contact structure are geodesics for different norms defined on the identity component of the group of compactly supported contactomorphisms and its universal cover. We characterize these geodesics by giving conditions on the Hamiltonian functions that generate them. For every norm considered we show that the norm of a contactomorphism that is the time-one of such a geodesic can be expressed in terms of the maximum of the absolute value of the corresponding Hamiltonian function. In particular we recover the fact that these norms are unbounded.
\end{abstract}

\section{Introduction}

\noindent
Hofer \cite{hofer} introduced in the 90's a conjugation invariant norm that comes from a Finsler structure on the group of compactly supported Hamiltonian symplectomorphisms of the standard symplectic Euclidean space $(\mathbb{R}^{2n},{\omega_{st}})$. This norm has then been generalized to any symplectic manifold by Lalonde and McDuff \cite{lalondemcduff}. The Hofer norm has been intensively studied (see for example \cite{bialypolterovich}, \cite{ostrover}, \cite{hoferzehnder}, \cite{lalondemcduff}, \cite{polterovich}, \cite{schwarz}, \cite{usher}, \cite{ustilovsky}) since the non-degeneracy and the geodesics of this norm are notions that are intimately linked to symplectic rigidity phenomena such as non-squeezing or Lagrangian intersection properties.\\

By contrast, in the contact setting Fraser, Polterovich and Rosen \cite{FPR} showed that there does not exist any conjugation invariant norm coming from a Finsler structure on the identity component of the group of compactly supported contactomorphisms of any contact manifold. More precisely they showed that any conjugation invariant norm on this group should be discrete, which means that there exists a positive constant such that any element that is not the identity has norm greater than this constant. In some sense this group is too big to carry a non-discrete conjugation invariant norm. Indeed, one important ingredient in their proof is an argument of contact flexibility: any Darboux ball can be contactly squeezed into an arbitrarily small one. \\

\noindent
Nevertheless, Sandon in \cite{sandonmetrique} constructed an integer-valued unbounded conjugation invariant norm on the identity component of the group of compactly supported contactomorphisms of $\mathbb{R}^{2n}\times S^1$ with its standard contact structure. Contact rigidity, in particular the existence of translated and discriminant points (see section \ref{section 3} for a definition), plays a crucial role for existence of such a norm. Indeed, if we forget the contact structure, Burago, Ivanov and Polterovich \cite{BIP} proved that there is no unbounded conjugation invariant norm on the identity component of the group of compactly supported diffeomorphism of $\mathbb{R}^{2n}\times S^1$. Since then several authors constructed different norms (conjugation invariant or not) on the identity component of the group of compactly supported contactomorphisms and on its universal cover \cite{discriminante}, \cite{FPR}, \cite{shelukhin}, \cite{Zapolsky}.\\

\noindent
The idea of this paper is to study the geodesics of some of these norms in this context. We focus our study on the discriminant norm \cite{discriminante}, the oscillation norm \cite{discriminante}, the Shelukhin norm \cite{shelukhin} and the Fraser-Polterovich-Rosen norm (FPR norm) \cite{FPR} on $\text{Cont}_0^c(\mathbb{R}^{2n}\times S^1,\xi_{st})$, the identity component of the group of compactly supported contactomorphisms of $\mathbb{R}^{2n}\times S^1$ endowed with its standard contact structure, and on its universal cover $\widetilde{\Cont_0^c}(\mathbb{R}^{2n}\times S^1,\xi_{st})$. As we recall below, for any co-oriented contact manifold, once a global contact form is fixed for the contact distribution, there is an explicit bijection between the space of smooth compactly supported time dependent functions on the contact manifold and the space of smooth paths of compactly supported contactomorphisms starting at the identity. While Shelukhin and Fraser, Polterovich and Rosen use this correspondence to define their norms in terms of the corresponding functions' size, the discriminant norm and the oscillation norm count in some sense the minimal number of discriminant points we should encounter while going from the identity to the considered contactomorphism along a smooth path.\\

Because the bijection mentioned above - between paths of contactomorphisms and functions - depends on the choice of a contact form, the Shelukhin norm is not conjugation invariant. All the other mentioned norms are conjugation invariant. However, as we will see in section \ref{section 2}, the Shelukhin norm, the oscillation norm and the discriminant norm share the common property to be defined by minimizing some length functionals on the space of smooth paths of compactly supported contactomorphisms. Therefore it makes sense to talk about the geodesics of these three norms: they are the paths that minimize the length (see section \ref{section 1} for a precise definition).\\ 

The main result that we present in this paper is the following: we show that a path of contactomorphisms generated by a function satisfying certain conditions is a geodesic for the discriminant, oscillation and Shelukhin norms. We show moreover that the norm of a contactomorphism that is the time-one of such a geodesic can be expressed in terms of the maximum of the absolute value of the corresponding function. In addition, even if the FPR norm does not \textit{a priori} come from a length functional - therefore we cannot talk about its geodesics - we still can express the FPR norm of such a time-one path in terms of the maximum of the absolute value of the corresponding function. In particular we get as a corollary a new proof of the unboundedness of the discriminant norm, the oscillation norm, the Shelukhin norm and the FPR norm on $\Cont_0^c(\mathbb{R}^{2n}\times S^1,\xi_{st})$. Before stating precisely these results let us define the standard contact structure and the standard contact form on $\mathbb{R}^{2n}\times S^1$. This will allow us to explicit the bijection between compactly supported time-dependent functions and compactly supported paths of contactomorphisms starting at the identity. \\ 

\noindent
For any positive natural number $n\in\mathbb{N}_{>0}$, the standard contact structure $\xi_{st}$ on $\mathbb{R}^{2n+1}$ is given by the kernel of the $1$-form $\alpha_{st}=dz-\sum\limits_{i=1}^ny_idx_i$, where $(x_1,\dots,x_n,y_1,\dots,y_n,z)$ are the coordinate
functions on $\mathbb{R}^{2n+1}$. The Reeb vector field $R_{\alpha_{st}}$ is given by $\frac{\partial}{\partial z}$. The $1$-form $\alpha_{st}$ is invariant by the action of $\mathbb{Z}$ on $\mathbb{R}^{2n+1}$ given by $k\cdot(x,y,z)=(x,y,z+k)$ for all $k\in\mathbb{Z}$ and $(x,y,z)\in\mathbb{R}^{2n+1}$, so it descends
to the quotient $\mathbb{R}^{2n}\times S^1$ to a $1$-form we also denote $\alpha_{st}$. The kernel of this latter form, which we still denote by $\xi_{st}$, is the standard contact distribution on $\mathbb{R}^{2n}\times S^1$, and the Reeb vector field $R_{\alpha_{st}}$ is again given by $\frac{\partial}{\partial z}$.\\

\noindent
To any compactly supported smooth path of contactomorphisms $\{\phi^t\}_{t\in[0,1]}\subset\text{Cont}_0^c(\mathbb{R}^{2n}\times S^1,\xi_{st})$ starting at the identity one can associate the compactly supported function $h:[0,1]\times \mathbb{R}^{2n}\times S^1\to\mathbb{R}$, $(t,p,z)\mapsto h^t(p,z)$ that satisfies
\[\alpha_{st}\left(\frac{d}{dt}\phi^t\right)=h^t\circ\phi^t\ \text{for all }t\in[0,1].\]
Conversely to any time dependent compactly supported smooth function $h :[0,1]\times\mathbb{R}^{2n}\times S^1\to\mathbb{R}$ one can associate the smooth path of vector fields $X_h : [0,1]\to\chi(M)$ satisfying 
\begin{equation}\label{contact vector field} \left\{
    \begin{array}{ll}
        \alpha_{st}(X_h^t)=h^t  \\
        d\alpha_{st}(X_h^t,\cdot)=dh^t\left(\frac{\partial}{\partial z}\right)\alpha_{st}-dh^t.
    \end{array}
\right.
\end{equation}
Using the Cartan's formula one can show that the flow $\{\phi_{X_h}^t\}_{t\in\mathbb{R}}$ of the time dependent vector field $X_h$, i.e.\ the path of diffeomorphisms such that $\frac{d}{dt}\phi_{X_h}^t=X_h^t(\phi_{X_h}^t)$ for all $t\in\mathbb{R}$ and $\phi_{X_h}^0=\Id$, is a smooth path of contactomorphisms. We denote by $\phi_h:=\{\phi_{X_h}^t\}_{t\in[0,1]}$ the restriction of this flow to the time interval $[0,1]$ and we say that $h$ is the Hamiltonian function that generates $\phi_h$. \\

In Theorem \ref{geodesique shelukhin/fpr} below $\nu_S^{\alpha_{st}}$ and $\widetilde{\nu_S^{\alpha_{st}}}$ denote the Shelukhin (pseudo-)norms on $\Cont_0^c(\mathbb{R}^{2n}\times S^1,\xi_{st})$ and $\widetilde{\Cont_0^c}(\mathbb{R}^{2n}\times S^1,\xi_{st})$ respectively, and the notation $\mathcal{L}_S^{\alpha_{st}}$ stands for the Shelukhin length functional that we define in section \ref{section 2}. We denote by $\nu^{\alpha_{st}}_{FPR}$ and $\widetilde{\nu^{\alpha_{st}}_{FPR}}$ the FPR norms on $\Cont_0^c(\mathbb{R}^{2n}\times S^1,\xi_{st})$ and $\widetilde{\Cont_0^c}(\mathbb{R}^{2n}\times S^1,\xi_{st})$ respectively.\\

\begin{theoreme}\label{geodesique shelukhin/fpr}

Let $H:\mathbb{R}^{2n}\to\mathbb{R}$ be a smooth function with compact support. If the Hessian of $H$ satifies $\underset{p\in\mathbb{R}^{2n}}\max \ |||\Hess_p(H)|||<2\pi$, where $|||\cdot|||$ denotes the operator norm, then the path $\phi_h$ generated by the Hamiltonian function 
\[\begin{aligned}
h:\mathbb{R}^{2n}\times S^1\to\mathbb{R}\ \ \ \ \ \ \
(p,z)&\mapsto H(p)
\end{aligned}\]
is a geodesic for the Shelukhin pseudo-norms $\widetilde{\nu_S^{\alpha_{st}}}$ and $\nu^{\alpha_{st}}_S$. More precisely we have
\[\mathcal{L}^{\alpha_{st}}_S(\phi_h)=\widetilde{\nu^{\alpha_{st}}_S}([\phi_h])=\nu^{\alpha_{st}}_S(\phi_h^1)=\max\{\max h;-\min h\}.\]
Moreover, for any compactly supported contactomorphism $\varphi\in\Cont_0^c(\mathbb{R}^{2n}\times S^1,\xi_{st})$ we have
\[\widetilde{\nu^{\alpha_{st}}_{FPR}}\left([\varphi\circ\phi_h\circ\varphi^{-1}]\right)=\nu^{\alpha_{st}}_{FPR}\left(\varphi\circ\phi_h^1\circ\varphi^{-1}\right)=\max\left\{\lceil\max h\rceil,\lceil -\min h\rceil\right\}.\]
\end{theoreme}

\vspace{3.ex}

In Theorem \ref{geodesique shelukhin/fpr} below $\nu_d$, $\widetilde{\nu_d}$ and $\nu_{osc}$, $\widetilde{\nu_{osc}}$ denote respectively the discriminant and oscillation norms on $\Cont_0^c(\mathbb{R}^{2n}\times S^1,\xi_{st})$ and $\widetilde{\Cont_0^c}(\mathbb{R}^{2n}\times S^1,\xi_{st})$ respectively. The notations $\mathcal{L}_d$ and $\mathcal{L}_{osc}$ stand respectively for the discriminant length functional and the oscillation length functional that we define in section \ref{section 2}.\\

\begin{theoreme}\label{geodesique discriminante/oscillation}

Let $H:\mathbb{R}^{2n}\to\mathbb{R}$ be a smooth function with compact support such that $0$ is a regular value of $H$ in the interior of its support, i.e for all $x\in\Interior\left(\Supp(H)\right)\cap H^{-1}\{0\},\ d_xH\ne 0$. Suppose that the Hessian of $H$ satisfies $\underset{p\in\mathbb{R}^{2n}}\max \ |||\Hess_p(H)|||<2\pi$, and consider the path $\phi_h$ generated by the Hamiltonian function 
\[\begin{aligned}
h:\mathbb{R}^{2n}\times S^1\to\mathbb{R}\ \ \ \ \ \ \
(p,z)&\mapsto H(p).
\end{aligned}\]
Then for any $\varphi\in\Cont_0^c(\mathbb{R}^{2n}\times S^1,\xi_{st})$ the path $\varphi\circ\phi_h\circ\varphi^{-1}$ is a geodesic for the discriminant norms $\widetilde{\nu_d}$ and $\nu_d$. More precisely we have
\[\mathcal{L}_d(\varphi\circ\phi_h\circ\varphi^{-1})=\widetilde{\nu_d}([\varphi\circ\phi_h\circ\varphi^{-1}])=\nu_d(\varphi\circ\phi_h^1\circ\varphi^{-1})=\max\{\lfloor \max h\rfloor+1,\lfloor -\min h\rfloor +1\}\ .\]
If we assume moreover that $H :\mathbb{R}^{2n}\to\mathbb{R}_{\geq 0}$ is non-negative (resp.\ $H :\mathbb{R}^{2n}\to\mathbb{R}_{\leq 0}$  is non-positive), then the path $\phi\circ\phi_h\circ\phi^{-1}$ is a geodesic for the oscillation norms $\widetilde{\nu_{osc}}$ and $\nu_{osc}$. More precisely
\[\begin{aligned}\mathcal{L}_{osc}(\varphi\circ\phi_h\circ\varphi^{-1})=\widetilde{\nu_{osc}}([\varphi\circ\phi_h\circ\varphi^{-1}])=\nu_{osc}(\varphi\circ\phi_h^1\circ\varphi^{-1})=&\lfloor \max h\rfloor+1\\
( \text{resp.\ }\mathcal{L}_{osc}(\varphi\circ\phi_h\circ\varphi^{-1})=\widetilde{\nu_{osc}}([\varphi\circ\phi_h\circ\varphi^{-1}])=\nu_{osc}(\varphi\circ\phi_h^1\circ\varphi^{-1})= &\lfloor-\min h\rfloor +1).\end{aligned}\]
\end{theoreme}

\vspace{3.ex}

\begin{corollaire}\label{corollaire 1}
The norms $\nu^{\alpha_{st}}_{FPR}$, $\nu^{\alpha_{st}}_S$, $\nu_d$, $\nu_{osc}$ are unbounded on $\Cont_0^c(\mathbb{R}^{2n}\times S^1,\xi_{st})$.
\end{corollaire}

\begin{proof}
Let $B>0$ be a positive number and consider the smooth compactly supported function 
\[\begin{aligned}
f_B :\mathbb{R} &\to\mathbb{R}_{\geq 0}\\
x &\mapsto 
     \left\{
    \begin{array}{ll}
        e^{-B^2(x^2-B)^{-2}} & \text{ if } x\in[-\sqrt{B},\sqrt{B}]  \\
        0 & \text{ if }|x|>\sqrt{B}.
    \end{array}
\right.
\end{aligned}\]
Denoting by $\lVert \cdot\rVert$ the standard Euclidean norm on $\mathbb{R}^{2n}$, for any $A>0$ the smooth non-negative, compactly supported Hamiltonian function 
\[h_{B,A} :\mathbb{R}^{2n}\times S^1\to\mathbb{R}_{\geq 0},\ \ \ \ \ \ (p,z)\mapsto \frac{A}{f_B(0)}f_B\left(\lVert p\rVert\right) \]
is independent of the $S^1$-coordinate and satisfies $h_{B,A}(0,z)=\max h_{B,A}=A$ for all $z\in S^1$. Note also that $h_{B,A}$ does not vanish inside the interior of its support. Moreover an easy computation shows that there exists an increasing function $B_0 :\mathbb{R}_{>0}\to\mathbb{R}_{>0}$ such that $\underset{p\in\mathbb{R}^{2n}}\max |||\Hess_p h_{B,A}(\cdot,z)|||<2\pi$ for any $B\geq B_0(A)$ and $z\in S^1$. So by Theorem \ref{geodesique shelukhin/fpr} and Theorem \ref{geodesique discriminante/oscillation}, the contactomorphism $\phi:=\phi_{h_{B_0(A),A}}^1$ satifies \[\nu_{osc}(\phi)=\nu_d(\phi)=\left\lfloor A\right\rfloor +1,\ \ \ \  \nu^{\alpha_{st}}_{FPR}(\phi)=\left\lceil A\right\rceil\ \  \text{ and }\ \  \nu^{\alpha_{st}}_{S}(\phi)=A.\] Since $A$ can be choosen as big as we want, we deduce Corollary \ref{corollaire 1}. Note that the diameter of the support of $h_{B_0(A),A}$ grows with $A$.\end{proof}

Unboundedness of these norms on $\Cont_0^c(\mathbb{R}^{2n}\times S^1,\xi_{st})$ was actually already known \cite{discriminante}, \cite{FPR}, \cite{shelukhin}. The novelty of our result is that given any positive number $A>0$ one can construct an explicit Hamiltonian function such that the norm of the time-one of the path generated by this Hamiltonian function is exactly $A$.\\

Even if these norms do not \textit{a priori} measure the same thing it turns out that they almost agree for the contactomorphisms described in Theorems \ref{geodesique shelukhin/fpr} and \ref{geodesique discriminante/oscillation}. It seems then reasonnable to ask whether these norms are equivalent. Similar questions can be found in \cite{FPR}, \cite{survey}.\\

Another natural question that arises from these results comes from the similarity with the geodesics of Hofer geometry. Indeed, it is well known that a time independent compactly supported function $H:\mathbb{R}^{2n}\to\mathbb{R}$ with small Hessian generates a path of Hamiltonian symplectomorphisms of the standard symplectic Euclidean space $(\mathbb{R}^{2n},{\omega_{st}})$ that is a geodesic for the Hofer norm \cite{bialypolterovich}, \cite{hoferzehnder}. So the above theorems say that these geodesics can be lifted to paths of contactomorphisms that are still geodesics for the Shelukhin norm, and under some more assumptions, geodesics for the discriminant norms and the oscillation norm too. It would be interesting to know which are the geodesics of the Hofer norm that lift to geodesics of $\nu_S^{\alpha_{st}}$, $\nu_d$ and $\nu_{osc}$.\\

The main tool to prove Theorems \ref{geodesique shelukhin/fpr} and \ref{geodesique discriminante/oscillation} is the translation selector $c :\text{Cont}_0^c(\mathbb{R}^{2n}\times S^1,\xi_{st})\to\mathbb{R}_{\geq 0}$ constructed by Sandon using generating functions \cite{sandonthese}. The strategy for the proof is to first bound all these norms from below in terms of this translation selector (see Corollary \ref{corollaire shelukhin} and Proposition \ref{crucial}) and in a second time to show that the paths we are considering realize this lower bound.\\

Before moving to the next section, it is interesting to note that Sandon in \cite{sandonthese} use this translation selector to associate to any domain $U\subset\mathbb{R}^{2n}\times S^1$ the following number
\[c(U):=\sup\left\{c(\phi)\ |\ \phi\in\Cont_0^c(U,\xi_{st})\right\}, \]
where $\Cont_0^c(U,\xi_{st})$ is the set of time-one maps of smooth paths whose supports are compacts which lie inside $U$. The integer part of this function turns out to be a contact capacity, more precisely we have
\begin{enumerate}
    \item $c(U)\leq c(V)$ for any $U\subset V$ lying inside $\mathbb{R}^{2n}\times S^1$, 
    \item $\left\lceil c(U)\right\rceil=\left\lceil c(\phi(U))\right\rceil$ for any $\phi\in\Cont_0^c(\mathbb{R}^{2n}\times S^1,\xi_{st})$ and any $U\subset\mathbb{R}^{2n}\times S^1$. 
    \item $c(\mathcal{B}^{2n}(r)\times S^1)=c(\mathcal{B}^2(r)\times\mathbb{R}^{2n-2}\times S^1)=\pi r^2$ where $\mathcal{B}^{2n}(r)$ and $\mathcal{B}^2(r)$ denote the standard Euclidean ball of radius $r$ centered at $0$ lying inside $\mathbb{R}^{2n}$ and $\mathbb{R}^2$ respectively.
\end{enumerate}

\noindent
In the same paper \cite{sandonthese} it is shown that if a contactomorphism $\phi\in\Cont_0^c(\mathbb{R}^{2n}\times S^1,\xi_{st})$ displaces $U\subset \mathbb{R}^{2n}\times S^1$, i.e.\ $\phi(U)\cap U=\emptyset$, then 
\begin{equation}\label{sandoninegalite}
\left\lceil c(U)\right\rceil\leq \left\lceil c(\phi)\right\rceil+\left\lceil c(\phi^{-1})\right\rceil.
\end{equation}

Corollary \ref{corollaire shelukhin} and Proposition \ref{crucial} will allow us to have similar results for the norms we are considering:

\begin{proposition}\label{capacite-energie}
Let $\phi\in\Cont_0^c(\mathbb{R}^{2n}\times S^1,\xi_{st})$ be a contactomorphism that displaces $U\subset\mathbb{R}^{2n}\times S^1$, i.e.\ $\phi(U)\cap U=\emptyset$. Then $\left\lceil \nu(\phi)\right\rceil \geq\frac{1}{2}\left\lceil c(U)\right\rceil $ where $\nu$ denotes the FPR, discriminant, oscillation or Shelukhin norm.\\ 
\end{proposition}

In the next section we give the basic definitions of the group of contactomorphisms, its universal cover,  conjugation invariant norms coming from length functionals and their geodesics. In the third section we recall the definition of the discriminant norm, the oscillation norm, the Shelukhin norm and the FPR norm. Then in the fourth section we give the main steps of the construction of Sandon's translation selector. Finally in the last section we give the proofs of Theorems \ref{geodesique shelukhin/fpr}, \ref{geodesique discriminante/oscillation} and Proposition \ref{capacite-energie}.\\

\textbf{Acknowledgement -} This project is a part of my PhD thesis done under the supervision of Sheila Sandon. I am very grateful to her for introducing me to this beautiful subject of research. I thank her for all the interesting discussions that we had without which this work would not have existed. I would also like to thank Miguel Abreu, Jean-François Barraud and Mihai Damian for their support and their useful advices.  The author is partially supported by the Deutsche Forschungsgemeinschaft under the Collaborative Research Center SFB/TRR 191 - 281071066 (Symplectic Structures in Geometry, Algebra and Dynamics).

\section{Basic definitions}\label{section 1}

\noindent
Let $M$ be a connected manifold of dimension $2n+1$ endowed with a co-oriented contact distribution $\xi$, i.e.\ a distribution of hyperplanes given by the kernel of a $1$-form $\alpha\in\Omega^1(M)$ such that $\alpha\wedge (d\alpha)^n$ is a volume form. We say that a diffeomorphism  $\phi\in\text{Diff}(M)$ is a contactomorphism if it preserves the contact distribution, i.e.\  $d_x\phi(\xi_x)=\xi_{\phi(x)}$ for all $x$ in $M$ or equivalently if there exists a smooth function $f: M\to\mathbb{R}\setminus\{0\}$ such that $\phi^*\alpha=f\alpha$. We say that a $[0,1]$-family of contactomorphisms $\{\phi^t\}_{t\in[0,1]}$ is a smooth path of contactomorphisms if the map $[0,1]\times M\to M$,
$(t,x)\mapsto \phi^t(x)$ is smooth. From now on we denote a smooth path of contactomorphisms by $\{\phi^t\}$ and omit the subscript $t\in[0,1]$. We will study the set of all contactomorphisms $\phi$ of $M$ that can be joined to the identity by a smooth path of compactly supported contactomorphisms $\{\phi^t\}$, that is to say that $\phi^0=\Id$, $\phi^1=\phi$ and $\Supp(\{\phi^t\}):=\overline{\underset{t\in[0,1]}\bigcup\{x\in M\ |\ \phi^t(x)\ne x\}}$ is a compact subset of $M$. This set is a group with respect to composition and we denote it by $\Cont^c_0(M,\xi)$. Another group that will be of interest is the universal cover of $\Cont^c_0(M,\xi)$: the set of classes of smooth compactly supported paths starting from the identity, where two different paths are identified if they are homotopic with fixed endpoints. More precisely, denoting by $\mathcal{C}^\infty([0,1],\Cont_0^c(M,\xi))$ the set of smooth paths of compactly supported contactomorphisms, we define $\widetilde{\Cont^c_0}(M,\xi):=\left\{\{\phi^t\}\in\mathcal{C}^\infty([0,1],\Cont_0^c(M,\xi))\ |\ \phi^0=\Id\right\} /\sim$, where $\{\phi^t\}\sim\{\varphi^t\}$ if and only if \begin{enumerate}
    
    \item they both finish at the same point, i.e.\ $\phi^1=\varphi^1$
    
    \item there exists a smooth function
    \[\begin{aligned}
    \Psi : [0,1]\times [0,1]\times M&\to M\\
    (s,t,x)&\mapsto \Psi_s^t(x)\ 
    \end{aligned}\]
    such that
    \begin{enumerate}
    \item for all $s\in[0,1]$, $\{\Psi_s^t\}_{t\in[0,1]}$ is a smooth path of contactomorphisms starting at identity and ending at $\phi^1=\varphi^1$
       
        \item for all $t\in[0,1]$, $\Psi_0^t=\phi^t$ and $\Psi_1^t=\varphi^t$.
    \end{enumerate}
\end{enumerate}

\noindent
The application 
\[\begin{aligned}
* : \widetilde{\Cont^c_0}(M,\xi)\times\widetilde{\Cont^c_0}(M,\xi)&\to\widetilde{\Cont^c_0}(M,\xi)\\
([\{\phi^t\}],[\{\varphi^t\}])&\mapsto [\{\phi^t\}]*[\{\varphi^t\}]:= [\{\phi^t\circ\varphi^t\}]
\end{aligned}\]
is well defined, and provides a group law on $\widetilde{\Cont^c_0}(M,\xi)$. Using the fact that time reparametrisation acts trivially on $\widetilde{\Cont^c_0}(M,\xi)$, one can show that this group law coincides with the concatenation of paths. More precisely let us fix a smooth bijective and increasing function $a : [0,1]\to[0,1]$ such that all the derivatives of $a$ at the time $0$ and $1$ vanish, and define for all smooth paths $\{\phi^t\}$, $\{\varphi^t\}$ contained in $\Cont_0^c(M,\xi)$ the concatenated path $\{\varphi^t\}\cdot\{\phi^t\}$ by
\[\{\varphi^t\}\cdot\{\phi^t\}:= \left\{
    \begin{array}{ll}
        \phi^{a(2t)} & \text{ if } t\in[0,1/2]  \\
        \varphi^{a(2t-1)}\circ\phi^1 & \text{ if }t\in[1/2,1].
    \end{array}
\right.\]
Then $[\{\phi^t\}\cdot\{\psi^t\}]=[\{\phi^t\}]*[\{\psi^t\}]$. Moreover for all $[\{\phi^t\}]$, $[\{\varphi^t\}]\in\widetilde{\Cont_0^c}(M,\xi)$ we have
\begin{enumerate}
    \item $[\{\phi^t\}]^{-1}=[\{\phi^{1-t}\circ(\phi^1)^{-1}\}]=[\{(\phi^t)^{-1}\}]$
    \item $[\{\varphi^t\}][\{\phi^t\}][\{\varphi^t\}]^{-1}=[\{\varphi^1\circ\phi^t\circ(\varphi^1)^{-1}\}]$.\\
\end{enumerate}

\begin{remarque}
     By putting the strong $\mathcal{C}^\infty$-topology on the group of compactly supported contactomorphisms $\Cont^c(M,\xi)$, the identity component of $\Cont^c(M,\xi)$ corresponds to $\Cont^c_0(M,\xi)$. Moreover this component is a sufficiently pleasant topological space so that its universal cover exists and can naturally be identified to $\widetilde{\Cont^c_0}(M,\xi)$. For more details we refer to the book of Banyaga \cite{banyagachemin}. \\
\end{remarque}

We will study in this paper four different norms on the groups $\text{Cont}_0^c(\mathbb{R}^{2n}\times S^1,\xi_{st})$ and $\widetilde{\Cont_0^c}(\mathbb{R}^{2n}\times S^1,\xi_{st})$. Three of them will be conjugation invariant.\\

\begin{definition}
Let $G$ be a group. We say that an application $\nu : G\to \mathbb{R}_{\geq 0}$ is a norm if for all $h,g\in G$ it satisfies:
\begin{enumerate}
    \item the triangular inequality, i.e.\ $\nu(gh)\leq \nu(g)+\nu(h)$ 
    \item symmetry, i.e.\ $\nu(g)=\nu(g^{-1})$
    \item non-degeneracy, i.e.\ $\nu(g)=0$ if and only if $g$ is the neutral element of the group.
\end{enumerate}
If $\nu$ satisfies only the first and second property we say that $\nu$ is a pseudo-norm. Moreover we say that $\nu$ is conjugation invariant if $\nu(hgh^{-1})=\nu(g)$.\\
\end{definition}

\begin{remarque}
  Another way to study conjugation invariant norms on a group $G$ is given by the point of view of bi-invariant metrics $d : G\times G\to\mathbb{R}_{\geq 0}$, i.e.\\ metrics $d$ such that $d(hg_1,hg_2)=d(g_1h,g_2h)=d(g_1,g_2)$ for all $g_1,g_2,h\in G$. Indeed to any conjugation invariant norm $\nu : G\to\mathbb{R}_{\geq 0}$ one can associate the bi-invariant metric $d(g_1,g_2):=\nu(g_1g_2^{-1})$. Conversely to any bi-invariant metric $d : G\times G\to\mathbb{R}_{\geq 0}$ one can associate the conjugation invariant norm $\nu(g)=d(\Id,g)$.\\
\end{remarque}

The norms we are interested in are of two types.\\

The first type of norms $\nu$ on $\Cont_0^c(M,\xi=\Ker(\alpha))$ and $\widetilde{\nu}$ on $\widetilde{\Cont_0^c}(M,\xi)$ we consider come from some length functional 
\[\mathcal{L} : \mathcal{C}^\infty([0,1],\Cont_0^c(M,\xi))\to\mathbb{R}_{\geq 0}\cup\{+\infty\}\] 
satisfying the following properties:
\begin{enumerate}
    \item $\mathcal{L}$ is invariant by time reparamatrisation, i.e.\ if $a :[0,1]\to[0,1]$ is a smooth bijective increasing function then $\mathcal{L}(\{\phi^t\})=\mathcal{L}(\{\phi^{a(t)}\})$
    \item $\mathcal{L}(\{\phi^t\}\cdot\{\psi^t\})\leq\mathcal{L}(\{\phi^t\})+\mathcal{L}(\{\psi^t\})$
    \item $\mathcal{L}(\{\phi^t\})>0$ for any path that is non constant
    \item $\mathcal{L}(\{\phi^{1-t}\circ(\phi^1)^{-1}\})=\mathcal{L}(\{\phi^t\})$ or $\mathcal{L}(\{(\phi^t)^{-1}\})=\mathcal{L}(\{\phi^t\})$.
\end{enumerate}

\noindent
The associated applications $\nu$ and $\widetilde{\nu}$ are then defined for any $\phi\in\Cont_0^c(M,\xi)\setminus\{\Id\}$ and for any  $[\{\phi^t\}]\in\widetilde{\Cont_0^c}(M,\xi)\setminus\{\Id\}$ by
\[\nu(\phi)=\inf\left\{\mathcal{L}(\{\phi^t\})\ |\ \phi^0=\Id\ , \phi^1=\phi)\right\}\ \ \ \text{ and }\ \ \  \widetilde{\nu}([\{\phi^t\}])=\inf\left\{\mathcal{L}(\{\varphi^t\})\ |\ [\{\varphi^t\}]=[\{\phi^t\}]\right\},\]
and $\nu(\Id)=\widetilde{\nu}(\Id)=0$.\\

Because of these properties of the length functional, the associated application $\nu$ (resp.\ $\widetilde{\nu})$ is automatically a pseudo-norm whenever $\nu(\phi)<+\infty$ for all $\phi\in\Cont_0^c(M,\xi)$ (resp.\ whenever $\widetilde{\nu}([\{\phi^t\}])<+\infty$ for all $[\{\phi^t\}]\in\widetilde{\Cont_0^c}(M,\xi)$).\\

The discriminant norms $\nu_d$, $\widetilde{\nu_d}$ and the Shelukhin norms $\nu^\alpha_S$,  $\widetilde{\nu_S^\alpha}$ defined on $\Cont_0^c(M,\xi)$ and $\widetilde{\Cont_0^c}(M,\xi)$), come respectively from the discriminant length functional $\mathcal{L}_d$ and the Shelukhin length functional $\mathcal{L}^\alpha_S$ that we will define in section \ref{section 2} below. For this type of norms we give the following definition of a geodesic.\\

\begin{definition}
A smooth compactly supported path of contactomorphisms $\{\phi^t\}$ starting at the identity is a geodesic for $\nu : \Cont_0^c(M,\xi)\to\mathbb{R}_{\geq 0}$ (resp.\ for $\widetilde{\nu} : \widetilde{\Cont_0^c}(M,\xi)\to\mathbb{R}_{\geq 0}$) if 
$\mathcal{L}(\{\phi^t\})=\nu(\phi^1)$ (resp.\ $\mathcal{L}(\{\phi^t\})=\nu([\{\phi^t\}])$).\\
\end{definition}

Following the terminology used by McDuff in \cite{geometricvariants}, the second type of norms $\nu$ on $\Cont_0^c(M,\xi)$ (resp.\ $\widetilde{\nu}$ on $\widetilde{\Cont_0^c}(M,\xi)$) we consider come from two seminorms $\widetilde{\nu_+}$ and $\widetilde{\nu_-}$ on $\widetilde{\Cont_0^c}(M,\xi)$ each one arising from some "semilength" functionals $\mathcal{L}_+, \mathcal{L}_-$. More precisely the functionals 
\[\mathcal{L}_\pm : \mathcal{C}^\infty([0,1],\Cont_0^c(M,\xi))\to\mathbb{R}_{\geq 0}\cup\{+\infty\}\]
verify the following properties:
\begin{enumerate}
    \item they are invariant by time reparametrisation, i.e.\ if $a :[0,1]\to[0,1]$ is a smooth bijective increasing function then $\mathcal{L}_\pm(\{\phi^t\})=\mathcal{L}_\pm(\{\phi^{a(t)}\})$
    \item $\mathcal{L}_\pm(\{\phi^t\}\cdot\{\psi^t\})\leq\mathcal{L}_\pm(\{\phi^t\})+\mathcal{L}_\pm(\{\psi^t\})$
      \item if $\{\phi^t\}$ is not the constant path then $\mathcal{L}_+(\{\phi^t\})>0$ or $\mathcal{L}_-(\{\phi^t\})>0$
    \item $\mathcal{L}_+(\{\phi^t\})=\mathcal{L}_-(\{\phi^{1-t}\circ(\phi^1)^{-1}\})$ or $\mathcal{L}_+(\{\phi^t\})=\mathcal{L}_-(\{(\phi^t)^{-1}\})$. 
\end{enumerate}
The seminorms $\widetilde{\nu_+}$ and $\widetilde{\nu_-}$ coming from these functionals are defined for any element $[\{\phi^t\}]\in\widetilde{\Cont_0^c}(M,\xi)$ by

\[\begin{aligned}
\widetilde{\nu_+}([\{\phi^t\}])&:=\inf\left\{\mathcal{L}_+\left(\{\varphi^t\}\right)\ |\ [\{\varphi^t\}]=[\{\phi^t\}]\right\}\\
\widetilde{\nu_-}([\{\phi^t\}])&:=-\inf\left\{\mathcal{L}_-\left(\{\varphi^t\}\right)\ |\ [\{\varphi^t\}]=[\{\phi^t\}]\right\}.
\end{aligned}\]

  From properties 2, 3 and 4 above, one can deduce that for any $[\{\phi^t\}]\in\widetilde{\Cont_0^c}(M,\xi)$ 
\begin{enumerate}
    \item $\widetilde{\nu_+}([\phi^t])=-\widetilde{\nu_-}([\{\phi^t\}]^{-1})$
    \item $\widetilde{\nu_+}\geq 0\geq \widetilde{\nu_-}$
    \item $\widetilde{\nu_+}([\{\phi^t\}][\{\psi^t\}])\leq\widetilde{\nu_+}([\{\phi^t\}])+\widetilde{\nu_+}([\{\psi^t\}])$\\
    $\widetilde{\nu_-}([\{\phi^t\}][\{\psi^t\}])\geq\widetilde{\nu_-}([\{\phi^t\}])+\widetilde{\nu_-}([\{\psi^t\}])$.\\
\end{enumerate}

\noindent
We then define the (pseudo-)norm $\widetilde{\nu}$ for all elements $[\{\phi^t\}]$ in $\widetilde{\Cont_0^c}(M,\xi)$ and the (pseudo-)norm $\nu$ for all elements $\phi$ in $\Cont_0^c(M,\xi)$ by
\[\widetilde{\nu}\left([\{\phi^t\}]\right)=\max\left\{\widetilde{\nu_+}\left([\{\phi^t\}]\right),-\widetilde{\nu_-}\left([\{\phi^t\}]\right)\right\}\ \ \text{ and }\ \ \nu(\phi)=\inf\left\{\widetilde{\nu}\left([\{\phi^t\}]\right)\ |\ \phi^1=\phi\right\}.\]
\noindent

Noticing that the functional $\mathcal{L}:=\max\{\mathcal{L}_+,\mathcal{L}_-\}$ is a genuine length functional and using again the terminology of McDuff \cite{geometricvariants} we give the following definition of a geodesic for this second type of norms.\\

\begin{definition}
We say that a smooth path of compactly supported contactomorphisms $\{\phi^t\}$ starting at the identity is a geodesic 
\begin{enumerate}
    \item for $\widetilde{\nu}$ if $\widetilde{\nu}\left([\{\phi^t\}]\right)=\mathcal{L}\left(\{\phi^t\}\right)$
    \item for $\nu$ if $\nu(\phi^1)=\mathcal{L}\left(\{\phi^t\}\right)$.\\
\end{enumerate}
\end{definition}

The oscillation norms $\nu_{osc}$, $\widetilde{\nu_{osc}}$ and the FPR norms $\nu^\alpha_{FPR}$, $\widetilde{\nu_{FPR}^\alpha}$ are defined via two seminorms. While the seminorms used to define the oscillation norm come from some semilength functionals, the seminorms of the FPR norms come from some functionals that are not invariant by time reparametrization. Therefore we are not going to talk about geodesics for the FPR norm.\\

\begin{remarque}
For all $[\{\phi^t\}]\in\widetilde{\Cont_0^c}(M,\xi)$ note that  $\widetilde{\nu}\left([\{\phi^t\}]\right)\leq\inf\left\{\mathcal{L}\left(\{\varphi^t\}\right)\ |\ [\{\varphi^t\}]=[\{\phi^t\}]\right\}$. This inequality has no reason \textit{a priori} to be an equality.

\end{remarque}

\section{The different norms}\label{section 2}
In this section we recall the definition of the discrimant norm \cite{discriminante}, the oscillation norm \cite{discriminante}, the Shelukhin norm \cite{shelukhin} and the FPR norm  \cite{FPR}.

\subsection{The discriminant norm}\label{sous section discriminant}

Let $\phi$ be a contactomorphism of $(M,\xi=\Ker(\alpha))$. A point $x$ in $M$ is said to be a discriminant point of $\phi$ if $\phi(x)=x$ and $(\phi^*\alpha)_x= \alpha_x$. This definition does not depend on the choice of the contact form. Moreover a point $x\in M$ is a discriminant point of a contactomorphism $\phi\in\Cont(M,\xi)$ if and only if:
\begin{enumerate}
    \item $x$ is a discriminant point of $\phi^{-1}$
    \item $\varphi(x)$ is a discriminant point of $\varphi\circ\phi\circ\varphi^{-1}$ for all $\varphi\in\Cont(M,\xi)$.\\
\end{enumerate}

\begin{remarque}
Another way to define the notion of discriminant point is to consider the symplectization of the contact manifold $(M,\xi)$. Let us denote by $\mathcal{S}_\xi(M):=\underset{x\in M}\bigcup \left\{\alpha\in T_x^*M,\ \Ker(\alpha)=\xi\right\}\subset T^*M$ the symplectization of $(M,\xi)$. It is a $(\mathbb{R}^*,\cdot)$-principal fiber bundle on $M$; the action of $\theta\in\mathbb{R}^*$ on any element $(x,\mu)\in\mathcal{S}_\xi(M)$ is given by $\theta\cdot(x,\mu)=(x,\theta\mu)$. Recall that $(\mathcal{S}_\xi(M),d\lambda_\xi)$ is a symplectic manifold, where $\lambda_\xi$ is the restriction of the canonical Liouville form $\lambda_M$ of $T^*M$ to $\mathcal{S}_\xi(M)$, i.e.\ $\lambda_M(q,p)(u)=p(d_{(q,p)}\pi(u))$ for all $(q,p)\in T^*M$, $u\in T_{(q,p)}T^*X$ and where $\pi : T^*X\to X$ is the canonical projection. To any contactomorphism $\phi\in\Cont_0(M,\xi)$ one can associate its $\mathbb{R}^*$-lift: the symplectomorphism $\psi$ of $(\mathcal{S}_\xi(M),d\lambda_\xi)$ defined by 
\[\psi(x,\mu)=(\phi(x),\mu\circ d_{\phi(x)}\phi^{-1})\ ,\ \forall (x,\mu)\in \mathcal{S}_\xi(M)\ .\]
So $x\in M$ is a discriminant point of a contactomorphism $\phi\in\Cont(M,\xi)$ if and only if any point $(x,\mu)\in \mathcal{S}_\xi(M)$ is a fixed point of its $\mathbb{R}^*$-equivariant lift $\psi$.\\
\end{remarque}

\noindent
For any contactomorphism $\phi\in\Cont_0^c(M,\xi)$ we denote by $DP(\phi)$ the set of all discriminant points of $\phi$. We say that a compactly supported smooth path of contactomorphisms $\{\phi^t\}_{t\in[a,b]}$ does not have discriminant points if
\begin{enumerate}
    \item $DP((\phi^s)^{-1}\circ\phi^t)=\emptyset$ for all $s\ne t\in[a,b]$, in the when case $M$ is compact
    \item $DP((\phi^s)^{-1}\circ\phi^t)\cap\Interior\left(\Supp(\{\phi^t\})\right)=\emptyset$ for all $s\ne t\in[a,b]$, in the case when $M$ is not compact.
\end{enumerate}

\noindent
The discriminant length of a smooth path of contactomorphisms will count in how many pieces we have to cut this path so that each piece does not have discriminant points. More precisely, the discriminant length of a non-constant smooth path of contactomorphisms $\{\phi^t\}$ of $(M,\xi)$ is defined by:
\[\begin{aligned}
\mathcal{L}_d(\{\phi^t\}):=\inf\{N\in\mathbb{N}^*\ |&\ \text{there exists } 0=t_0<...<t_N=1, \text{ such that for all } i\in [0,N-1]\cap\mathbb{N}\\
&\ \{\phi^t\}_{t\in[t_i,t_{i+1}]}\ \text{does not have discriminant point} \}\in\mathbb{N}_{>0}\cup\{+\infty\}.
\end{aligned}\]
By convention we set $\inf\emptyset=+\infty$.\\

\noindent
Because of the two properties of discriminant points mentioned at the beginning of this subsection, it is straightforward that the functional $\mathcal{L}_d$ is a length functional which is invariant under conjugation of elements in $\Cont_0^c(M,\xi)$. The discriminant norms $\widetilde{\nu_d}$ on $\widetilde{\Cont_0^c}(M,\xi)$ and $\nu_d$ on $\Cont_0^c(M,\xi)$ are then defined as follow :
\[\begin{aligned}
\widetilde{\nu_d}\left([\{\phi^t\}]\right)&=\min\left\{\mathcal{L}_d(\{\varphi^t\}) \ |\ [\{\varphi^t\}]=[\{\phi^t\}]\right\},\ \forall [\{\phi^t\}]\in\widetilde{\Cont_0^c}(M,\xi)\setminus\{\Id\} \text{ and } \widetilde{\nu_d}(\Id)=0\\
\nu_d(\phi)&=\min\left\{\mathcal{L}_d(\{\phi^t\})\ |\ \phi^0=\Id \text{ and } \phi^1=\phi \right\}, \forall \phi\in\Cont_0^c(M,\xi)\setminus\{\Id\}\text{ and } \nu_d(\Id)=0.
\end{aligned}\]
 \textit{A priori} the applications $\nu_d$ and $\widetilde{\nu_d}$ may take the value $+\infty$, but this is not the case. For a proof of this statement we refer the reader to \cite{discriminante}. So for any co-oriented contact manifold $(M,\xi)$ the applications $\nu_d$ and $\widetilde{\nu_d}$
are well defined conjugation invariant norms.\\ 

\noindent
Another way to define the norms $\widetilde{\nu_d}$ and $\nu_d$ is to see them as word norms where the generating sets are respectively 
\[\begin{aligned}\widetilde{\mathcal{E}}&:=\left\{[\{\phi^t\}]\in\widetilde{\Cont_0^c}(M,\xi)\ |\ \{\phi^t\} \text{ does not have discriminant points}\right\} \text{ and }\\
 \mathcal{E}&:=\Pi(\widetilde{\mathcal{E}}) 
\end{aligned}\]
where $\Pi :\widetilde{\Cont_0^c}(M,\xi)\to\Cont_0^c(M,\xi)$ is the natural projection.

\subsection{The oscillation norm}
 Following Eliashberg and Polterovich \cite{EP00}, we say that a smooth path $\{\phi^t\}_{t\in[a,b]}\subset\Cont^c_0(M,\xi)$ is positive (resp.\ negative) if \begin{enumerate}
     \item $\alpha_{\phi^t(x)}\left(\frac{d}{dt}\phi^t(x)\right)> 0\ (\text{resp.\ } <0),\ \text{ for all } t\in[a,b],\text{ and for all } x\in M$, when $M$ is compact
     \item $\alpha_{\phi^t(x)}\left(\frac{d}{dt}\phi^t(x)\right)> 0\ (\text{resp.\ } <0),\ \text{ for all } t\in[a,b],\text{ and for all } x\in \Interior\left(\Supp\left(\{\phi^t\}\right)\right)$, when $M$ is non compact.
 \end{enumerate}
 
 \noindent
 Positivity or negativity of a path does not depend on the choice of the contact form but only on the choice of the co-orientation of $\xi$. One can also define the notion of non-negative (resp.\ non-positive) smooth path, by replacing the above strict inequalities with large inequalities. \\
 
If there exists a non-constant and non-negative contractible loop of compactly supported contactmorphisms then the contact manifold $(M,\xi)$ is called non universally orderable, in the other case we say that $(M,\xi)$ is universally orderable. When $(M,\xi)$ is universally orderable, Eliashberg and Polterovich in \cite{EP00} defined a bi-invariant partial order on $\widetilde{\Cont_0^c}(M,\xi)$. For more details we refer the reader to \cite{EP00}, \cite{chernovnemirovski2} and to the remark \ref{compatibilite avec la relation d'ordre} below.\\

\begin{remarque}
To prove that a contact manifold is orderable involves some "hard" symplectic/contact techniques. See for instance in \cite{chernovnemirovski1}, \cite{chernovnemirovski2},  \cite{CFP}, \cite{EKP}, \cite{EP00} pseudo-holomorphic curves or generating functions techniques are used to prove the universal orderability of the unitary cotangent bundle and the $1$-jet bundle of any compact manifold. We will see below (Remark \ref{remarque sur ordonnabilite}) how Bhupal \cite{bhupal} and Sandon \cite{sandonthese} show that $(\mathbb{R}^{2n+1},\xi_{st})$ and $(\mathbb{R}^{2n}\times S^1,\xi_{st})$ are universally orderable using also generating functions techniques.
\end{remarque}

\begin{definition} Let $\{\phi^t\}$ be a non-constant smooth path in $\Cont_0^c(M,\Ker(\alpha))$. We define the  semilength functionals used to construct the oscillation norm by
\[\begin{aligned}
\mathcal{L}_+^{osc}(\{\phi^t\}):=\inf\{ N\in\mathbb{N}^*\ |&\ \text{ there exist } K\in\mathbb{N}^*, K\geq N,\ \text{ and }\  0=t_0<...<t_K=1 \text{ so that}\\
&\ \text{ Card}\{i\in[0,K-1]\cap\mathbb{N}\ |\ \{\phi^t\}_{t\in]t_i,t_{i+1}[}\ \text{ is positive}\}=N\\
&\ \text{ Card}\{i\in[0,K-1]\cap\mathbb{N}\ |\ \{\phi^t\}_{t\in]t_i,t_{i+1}[}\ \text{ is negative}\}=K-N\\
&\ \ \mathcal{L}_d\left(\{\phi^t\}_{t\in[t_i,t_{i+1}]}\right)=1,\ \text{ for all } i\in[0,K-1]\cap\mathbb{N}\}
\end{aligned}\]

\[\begin{aligned}
\mathcal{L}_-^{osc}(\{\phi^t\}):=\inf\{ N\in\mathbb{N}^*\ |&\ \text{ there exist } K\in\mathbb{N}^*,\ K\geq N,\ \text{ and }\ 0=t_0<...<t_K=1 \text{ so that}\\
&\ \text{ Card}\{i\in[0,K-1]\cap\mathbb{N}\ |\ \{\phi^t\}_{t\in]t_i,t_{i+1}[}\ \text{ is positive}\}=K-N\\
&\ \text{ Card}\{i\in[0,K-1]\cap\mathbb{N}\ |\ \{\phi^t\}_{t\in]t_i,t_{i+1}[}\ \text{ is negative}\}=N\\
&\ \ \mathcal{L}_d\left(\{\phi^t\}_{t\in[t_i,t_{i+1}]}\right)=1,\ \text{ for all } i\in[0,K-1]\cap\mathbb{N}\}.
\end{aligned}\]
By convention we set $\inf\emptyset=+\infty$. \\
\end{definition}

\noindent
The oscillation length of a smooth path $\{\phi^t\}$ is then by definition \[\mathcal{L}_{osc}(\{\phi^t\})=\max\{\mathcal{L}_+^{osc}(\{\phi^t\}),\mathcal{L}_-^{osc}(\{\phi^t\})\}\] and the corresponding seminorms are
\[\widetilde{\nu_\pm^{osc}}\left([\{\phi^t\}]\right):=\pm\inf\left\{\mathcal{L}_\pm^{osc}\left(\{\varphi^t\}\right)\ |\ [\{\varphi^t\}]=[\{\phi^t\}]\right\} \text{ for all } [\{\phi^t\}]\in\widetilde{\Cont_0}(M,\xi).\]
As discussed in the previous section, to these seminorms we associate the application $\widetilde{\nu_{\text{osc}}}:=\max\left\{\widetilde{\nu_+},-\widetilde{\nu_-}\right\}$ on $\widetilde{\Cont_0^c}(M,\xi)$ and the application $\nu_{\text{osc}}$ on $\Cont_0^c(M,\xi)$ defined by
\[\nu_{\text{osc}}(\phi)=\inf\{\widetilde{\nu_{\text{osc}}}\left([\{\phi^t\}])\ | \phi^1=\phi\right\}\ \text{ for all }\phi\in\Cont_0^c(M,\xi).\]

Colin and Sandon in \cite{discriminante} showed that the applications $\nu_{\text{osc}}$ and $\widetilde{\nu_{\text{osc}}}$ are well defined conjugation invariant norms on $\Cont_0^c(M,\xi)$ and on $\widetilde{\Cont_0^c}(M,\xi)$) respectively if and only if $(M,\xi)$ is universally orderable.\\
\noindent

\begin{remarque}\label{compatibilite avec la relation d'ordre}
An interesting property of the norm $\widetilde{\nu_{osc}}$ is the compatibility with the bi-invariant partial order $\succeq$ defined on $\widetilde{\Cont_0^c}(M,\xi)$ by Eliashberg and Polterovich in \cite{EP00} when $(M,\xi)$ is an universally orderable contact manifold.  More precisely for any $[\{\varphi^t\}]\succeq[\{\phi^t\}]\succeq \Id$ we have $\widetilde{\nu_{osc}}\left([\{\varphi^t\}]\right)\geq\widetilde{\nu_{osc}}\left([\{\phi^t\}]\right)$.
\end{remarque}

\subsection{The Shelukhin norm}

The Shelukhin length of a smooth path of compactly supported contactomorphisms $\{\phi^t\}_{t\in[0,1]}$ in $\Cont_0^c(M,\xi=\Ker(\alpha))$ is defined by 
\[\mathcal{L}^\alpha_S\left(\{\phi^t\}\right)=\int_0^1\underset{x\in M}\max \left|\alpha_{\phi^t(x)}\left(\frac{d}{dt}\phi^t(x)\right)\right|dt.\]

\noindent
To these length functionals we associate the applications
\[\begin{aligned}
\widetilde{\nu^\alpha_S}\left([\{\phi^t\}]\right)&=\inf\left\{\mathcal{L}^\alpha_S(\{\varphi^t\}) \ |\ [\{\varphi^t\}]=[\{\phi^t\}]\right\}\ \text{ and}\\
\nu^\alpha_S(\phi)&=\inf\left\{\mathcal{L}^\alpha_S(\{\phi^t\})\ |\ \phi^0=\Id \text{ and } \phi^1=\phi \right\}.
\end{aligned}\]

The application $\widetilde{\nu_S^\alpha}$ is a pseudo-norm, and Shelukhin proved in \cite{shelukhin} that the application $\nu_S^\alpha$ is a norm. None of them is conjugation invariant, indeed for all $\varphi\in\Cont_0^c(M,\xi)$ and for all $[\{\phi^t\}]\in\widetilde{\Cont_0^c}(M,\xi)$ we have
\[\widetilde{\nu_S^\alpha}\left([\{\varphi\circ\phi^t\circ\varphi^{-1}]\right)=\widetilde{\nu_S^{\varphi^*\alpha}}\left([\{\phi^t\}]\right)\ \ \text{ and }\ \ \nu_S^\alpha\left(\varphi\circ\phi^1\circ\varphi^{-1}\right)=\nu_S^{\varphi^*\alpha}\left(\phi^1\right).\]

\vspace{2.ex}

\noindent
The fact that $\nu_S^\alpha$ is non-degenerate is non-trivial. It is proved by Shelukhin in \cite{shelukhin} using an energy-capacity inequality: if $\phi\in\Cont_0^c(M,\xi)$ is not the identity then it displaces a ball, and its norm is greater than the capacity of this ball. Since the same argument cannot be applied to loops of contactomorphisms based at the identity, it may exist $[\{\phi^t\}]\in \pi_1\left(\Cont_0^c(M,\xi)\right)\setminus\{\Id\}$ such that $\widetilde{\nu_S^\alpha}\left([\{\phi^t\}]\right)=0$. For more details we refer the reader to \cite{shelukhin}.

\subsection{The FPR norm} 
As the oscillation norm, the FPR norm \cite{FPR} comes from two seminorms $\widetilde{\nu^{\alpha}_\pm}$ and is well defined for a contact manifold $(M,\xi=\Ker(\alpha))$ that is universally orderable. The seminorms are defined for any $[\{\phi^t\}]\in\widetilde{\Cont_0^c}(M,\xi)$ by
\[\begin{aligned}
\widetilde{\nu^\alpha_+}([\{\phi^t\}])&=\min\left\{\left\lceil \underset{t,x}\max\ \alpha_{\varphi^{t}(x)}\left(\frac{d}{dt}\varphi^{t}(x)\right)\right\rceil\ |\  [\{\varphi^t\}]=[\{\phi^t\}]\in\widetilde{\Cont_0^c}(M,\xi)\right\}\text{ and }\\
\widetilde{\nu^\alpha_-}([\{\phi^t\}])&=-\min\left\{\left\lceil -\underset{t,x}\min\ \alpha_{\varphi^{t}(x)}\left(\frac{d}{dt}\varphi^{t}(x)\right)\right\rceil\ |\   [\{\varphi^t\}]=[\{\phi^t\}]\in\widetilde{\Cont_0^c}(M,\xi)\right\} .\end{aligned}\] 

The applications $\widetilde{\nu^\alpha_{\text{FPR}}}=\max\left\{\widetilde{\nu_+^\alpha},-\widetilde{\nu_-^\alpha}\right\}$ and $\nu^\alpha_{\text{FPR}} : \phi\mapsto\inf\left\{\widetilde{\nu^\alpha_{\text{FPR}}}\left([\{\phi^t\}]\right)\ |\ \phi^1=\phi\right\}$ are well defined norms on $\widetilde{\Cont_0^c}(M,\xi)$ and on $\Cont_0^c(M,\xi)$ respectively.\\

Moreover Fraser, Polterovich and Rosen showed in \cite{FPR} that if the Reeb flow associated to the contact form $\alpha$ is periodic then $\widetilde{\nu^\alpha_{FPR}}$ and $\nu^\alpha_{FPR}$ are conjugation invariant norms. This comes from the fact that the fundamental group of $\Cont_0^c(M,\xi)$ is included in the center of $\widetilde{\Cont}_0^c(M,\xi)$ (they are actually equal).\\

\begin{remarque}
  The norm $\widetilde{\nu_{FPR}^\alpha}$ on $\widetilde{\Cont_0^c}(M,\xi)$ is also compatible with the partial order: $\widetilde{\nu_{FPR}^\alpha}([\{\phi^t\}])\geq \widetilde{\nu_{FPR}^\alpha}([\{\varphi^t\}])$ for any $[\{\phi^t\}]\succeq[\{\varphi^t\}]\succeq\Id$. \\
\end{remarque}

\noindent
Because the functionals
\[\begin{aligned}
\mathcal{C}^\infty([0,1],\Cont_0^c(M,\xi))&\to\mathbb{R}\\
\{\phi^t\}&\mapsto \left\lceil \underset{t,x}\max\ \alpha_{\phi^{t}(x)}\left(\frac{d}{dt}\phi^{t}(x)\right)\right\rceil
\end{aligned}\]
\[\begin{aligned}
\mathcal{C}^\infty([0,1],\Cont_0^c(M,\xi))&\to\mathbb{R}\\
\{\phi^t\}&\mapsto \left\lceil \underset{t,x}-\min\ \alpha_{\phi^{t}(x)}\left(\frac{d}{dt}\phi^{t}(x)\right)\right\rceil
\end{aligned}\]
are not invariant under time reparametrization we cannot talk about semilength functionals, and so we are not going to talk about the geodesics of the FPR norm.

\section{Generating functions and translation selector}\label{section 3}

An important ingredient to prove Theorems \ref{geodesique shelukhin/fpr} and \ref{geodesique discriminante/oscillation} is to be able to detect translated points of contactomorphisms and their translations. To do so and to avoid confusion we will sometimes see contactomorphisms of $(\mathbb{R}^{2n}\times S^1,\xi_{st})$ as $1$-periodic contactomorphisms of $(\mathbb{R}^{2n+1},\xi_{st})$. More precisely, let $\Cont_0^c(\mathbb{R}^{2n+1},\xi_{st})^{1-per}$ be the group of contactomorphisms $\phi$ of $(\mathbb{R}^{2n+1},\xi_{st})$ that can be joined to the identity by a smooth path $\{\phi^t\}$ such that 
\begin{enumerate}
    \item $\Supp(\{\phi^t\})$ is contained in $K\times\mathbb{R}$, where $K$ is a compact subset of $\mathbb{R}^{2n}$
    \item $\phi^t(x,y,z+k)=\phi^t(x,y,z)+(0,0,k)$ for any $t\in[0,1]$ and for any $k\in\mathbb{Z}$. 
    \end{enumerate}
    The natural projection $\Cont_0^c(\mathbb{R}^{2n+1},\xi_{st})^{1-per}\to\Cont_0^c(\mathbb{R}^{2n}\times S^1,\xi_{st})$ is a group isomorphism whose inverse is given by lifting. \\

\begin{definition}[\cite{sandonthese}]
Let $\phi\in\Cont_0^c(\mathbb{R}^{2n}\times S^1,\xi_{st})$ be a contactomorphism and denote by $\widetilde{\phi}\in \Cont_0^c(\mathbb{R}^{2n+1},\xi_{st})^{1-per}$ its lift. A point $(p,[z])\in\mathbb{R}^{2n}\times S^1$ is a $t$-translated point for $\phi$ if $\widetilde{\phi}(p,z)=(p,z+t) \text{ and } g(p,[z])=0$
where $(p,z)\in\mathbb{R}^{2n+1}$ is any point that projects on $(p,[z])$ and $g$ denotes the conformal factor of $\phi$ with respect to $\alpha_{st}=dz-\sum\limits_{i=1}^ny_idx_i$, i.e.\ $\phi^*\alpha_{st}=e^{g}\alpha_{st}$. The spectrum of $\phi$ is then defined by
\[\text{Spectrum}(\phi):=\left\{ t\in\mathbb{R}\ |\ \phi \text{ has a } t\text{-translated point}\right\}.\]
\end{definition}

\vspace{2.ex}

\begin{remarque}\label{remarque2}
\begin{enumerate}
    \item The definition of translated points depends on the choice of a contact form.
    \item Let $(p,[z])\in\mathbb{R}^{2n}\times S^1$ be a $k$-translated point of $\phi\in\Cont_0^c(\mathbb{R}^{2n}\times S^1,\xi_{st})$ with $k\in\mathbb{Z}$. Then $(p,[z])$ is a discriminant point of $\phi$. In particular, $\mathbb{Z}$-translated points do not depend on the choice of the contact form. In the same spirit we deduce that for all $\phi,\psi\in\Cont_0^c(\mathbb{R}^{2n}\times S^1,\xi_{st})$  \[\text{Spectrum}(\phi)\cap\mathbb{Z}=\text{Spectrum}(\psi\circ\phi\circ\psi^{-1})\cap\mathbb{Z}.\]
    \end{enumerate}
\end{remarque}

\vspace{2.ex}
\noindent
Sandon in \cite{sandonthese} constructed a function 
\[c :\Cont_0^c(\mathbb{R}^{2n}\times S^1,\xi_{st})\to\mathbb{R}_{\geq 0}\ \ \ \ \ \ \phi\mapsto c(\phi)\in\text{Spectrum}(\phi)\ \]
 that satisfies algebraic and topological properties that we list at the end of this section in Theorem \ref{translation selector} and that will be intensively used for the proofs of Theorem \ref{geodesique shelukhin/fpr} and Theorem \ref{geodesique discriminante/oscillation}. We call this function a translation selector. The rest of this section will be devoted to give the main steps of its construction that we will need for the last section concerning the proofs of Theorem \ref{geodesique shelukhin/fpr}, Theorem \ref{geodesique shelukhin/fpr} and Proposition \ref{inégalité géodésique shelukhin}. For this purpose we will follow mainly \cite{sandonthese}.

\subsection{The graph of a contactomorphism as a compact Legendrian of the 1-jet bundle}\label{section graphe}

The aim of this paragraph is to associate to any contactomorphism $\phi\in\Cont_0^c(\mathbb{R}^{2n}\times S^1,\xi_{st})$ a Legendrian $\Lambda_\phi$ of the $1$-jet bundle of $S^{2n}\times S^1$ in such a way that there is a one-to-one correspondence between translated points of $\phi$ and intersections of $\Lambda_\phi$ with the $0$-wall of the $1$-jet bundle. We give a rather explicit and detailed description of this construction since for proving the estimate of Proposition \ref{inégalité géodésique shelukhin} below it will be important to make sure that all the maps involved in the process of this construction preserve not only the contact structures but also the contact forms.\\

First recall that for a smooth manifold $X$ the $1$-jet bundle of $X$ is the manifold $J^1X:=T^*X\times\mathbb{R}$, where $T^*X$ is the cotangent bundle of $X$. This space carries a canonical contact structure $\xi^X$ given by the kernel of the $1$-form $\alpha_X:=dz-\lambda_X$ where $z$ is the coordinate function on $\mathbb{R}$ and $\lambda_X$ is the Liouville form of the cotangent bundle $T^*X$. We denote by $\mathbb{O}_X:=\{(q,0)\in T^*X\ |\ q\in X\}$ the $0$-section of the cotangent bundle $T^*X$ and refer to $\mathbb{O}_X\times\mathbb{R}$ as the $0$-wall of $J^1X$.\\

\noindent
Let $\phi\in\Cont_0^c(\mathbb{R}^{2n}\times S^1,\xi_{st})$ be a compactly supported contactomorphism and consider its $1$-periodic correspondent $\widetilde{\phi}\in\Cont_0^c(\mathbb{R}^{2n+1},\xi_{st})^{1-per}$. We write $g$ and $\widetilde{g}$ for their conformal factors with respect to $\alpha_{st}$, i.e.\ $\phi^*\alpha_{st}=e^g\alpha_{st}$ and $\widetilde{\phi}^*\alpha_{st}=e^{\widetilde{g}}\alpha_{st}$. Then the image of the map
\[\text{gr}_{\alpha_{st}}(\widetilde{\phi}) : \mathbb{R}^{2n+1}\to\mathbb{R}^{2n+1}\times\mathbb{R}^{2n+1}\times\mathbb{R}
\ \ \ \ \ \ \ \ p\mapsto (p,\widetilde{\phi}(p),\widetilde{g}(p))\]
is a Legendrian of $\mathbb{R}^{2n+1}\times\mathbb{R}^{2n+1}\times\mathbb{R}$
endowed with the contact $1$-form $\beta:=\alpha_{st}^2-e^\theta\alpha_{st}^1$, where $\theta$ is the coordinate function on $\mathbb{R}$ and where $\alpha_{st}^i$ for $i\in\{1,2\}$ is the pull back of $\alpha_{st}$ by the projection  $pr_i(p_1,p_2,\theta)=p_i$.\\

\noindent
The application 
\[\begin{aligned}
\Theta : \mathbb{R}^{2n+1}\times\mathbb{R}^{2n+1}\times\mathbb{R}&\hookrightarrow J^1(\mathbb{R}^{2n+1})\\
(x,y,z,X,Y,Z,\theta)&\mapsto (x,Y,z,Y-e^\theta y,x-X,e^\theta-1,xY-XY+Z-z)
\end{aligned}\]
is an exact contact embedding of $((\mathbb{R}^{2n+1})^2\times\mathbb{R},\Ker(\beta))$ in $(J^1(\mathbb{R}^{2n+1}),\Ker(\alpha_{\mathbb{R}^{2n+1}}))$, i.e.\ $\Theta^*(dz-\lambda_{\mathbb{R}^{2n+1}})=\beta$. Denote by $\Lambda_{\widetilde{\phi}}^{\mathbb{R}^{2n+1}}:=\Theta\circ\text{gr}_{\alpha_{st}}(\widetilde{\phi})$. Then the following lemma is a direct consequence of the definition of $\Theta$.\\

\begin{lemme}\label{lemme1}
\begin{enumerate}
    \item $\Lambda^{\mathbb{R}^{2n+1}}_{\Id}$ is the $0$-section of  $J^1(\mathbb{R}^{2n+1})$.
    \item If $\{\phi^t\}\subset\Cont_0^c(\mathbb{R}^{2n}\times S^1,\xi_{st})$ is a smooth path of contactmorphisms then $\left\{\text{Image}\left(\Lambda_{\widetilde{\phi^t}}^{\mathbb{R}^{2n+1}}\right)\right\}$ is a smooth path of Legendrians.
    \item For any $\phi\in\Cont_0^c(\mathbb{R}^{2n}\times S^1,\xi_{st})$, there is a one-to-one correspondence between the set of $t$-translated points of $\phi$ and $\Lambda_{\widetilde{\phi}}^{\mathbb{R}^{2n+1}}\left(\mathbb{R}^{2n+1}\right)\bigcap \mathbb{O}_{\mathbb{R}^{2n+1}}\times\{t\}$.\\
    \end{enumerate}
\end{lemme}

\noindent
The aim now is to be able to replace the Euclidean space $\mathbb{R}^{2n+1}$ of the previous lemma with a compact manifold without loosing the 3 properties above. This will be done in two steps. First using the periodicity we will replace $\mathbb{R}^{2n+1}$ by $\mathbb{R}^{2n}\times S^1$. Then we will compactify the $\mathbb{R}^{2n}$ coordinates into the standard Euclidean sphere $S^{2n}\subset \mathbb{R}^{2n+1}$. \\

\noindent
Because $\phi$ is $1$-periodic in the
$z$-direction, the map 
\[\Lambda_\phi^{\mathbb{R}^{2n}\times S^1} : \mathbb{R}^{2n}\times S^1\to J^1(\mathbb{R}^{2n}\times S^1)\ \ \ \ \ \ (p,[z])\mapsto \text{pr}\left(\Lambda_{\widetilde{\phi}}^{\mathbb{R}^{2n+1}}(p,z)\right) ,\]
where $z\in\mathbb{R}$ is any representant of $[z]\in S^1=\mathbb{R}/\mathbb{Z}$ and $\text{pr} :J^1(\mathbb{R}^{2n+1})\to J^1(\mathbb{R}^{2n}\times S^1)$ the natural projection, is well defined. Note that the projection $\text{pr}$ is a covering map that preserves the contact forms, i.e.\ $\text{pr}^*(\alpha_{\mathbb{R}^{2n}\times S^1})=\alpha_{\mathbb{R}^{2n+1}}$, and that the map $\Lambda_\phi^{\mathbb{R}^{2n}\times S^1}$ enjoys again the 3 properties of Lemma \ref{lemme1}. \\

\noindent
Finally, fixing $p_0$ a point on $S^{2n}$ the stereographic projection $\psi : S^{2n}\setminus\{p_0\}\to\mathbb{R}^{2n}$ gives a diffeomorphism $\overline{\psi}:=\psi\times\Id : S^{2n}\setminus\{p_0\}\times S^1\to\mathbb{R}^{2n}\times S^1$ that lifts to the strict contactomorphism 
\[\begin{aligned}
\Psi :J^1(S^{2n}\setminus\{p_0\}\times S^1)&\to J^1(\mathbb{R}^{2n}\times S^1)\\
(x,\mu,z)&\mapsto (\overline{\psi}(x),\mu\circ d_{\overline{\psi}(x)}\overline{\psi}^{-1},z).
\end{aligned}\]
Since $\phi$ is compactly supported in the $\mathbb{R}^{2n}$-direction, the map

\[\Lambda_\phi : S^{2n}\times S^1\ \ \ \ \ (p,z)\mapsto \left\{
    \begin{array}{ll}
        \Psi^{-1}\left(\Lambda_\phi^{\mathbb{R}^{2n}\times S^1}(\overline{\psi}(p,z))\right) & \mbox{if } p\ne p_0 \\
        ((p,z),0,0) & \mbox{ if } p=p_0
    \end{array}
    \right.\]
is a smooth Legendrian embedding that still enjoys the three properties of Lemma \ref{lemme1}. Moreover now this Legendrian is compact.\\

In the case when $\phi$ is $\mathcal{C}^1$-small, $\Lambda_\phi$ is a Legendrian section of $J^1(S^{2n}\times S^1)$ and so it is given by the $1$-jet of a function $f:S^{2n}\times S^1\to\mathbb{R}$, i.e.\ $\Lambda_\phi(x)=(x,d_xf,f(x))$ for all $x\in S^{2n}\times S^1$. In particular there is a one-to-one correspondence between critical points of $f$ of critical value $t$ and $t$-translated points of $\phi$. So when $\phi\in\Cont_0^c(\mathbb{R}^{2n+1},\xi_{st})^{z-per}$ is $\mathcal{C}^1$-small, looking for translated points of $\phi$ is equivalent to looking for critical point of $f$. For the latter problem, Morse theory can be applied to ensure existence of critical points. \\

For a general contactomorphism $\phi\in\Cont_0^c(\mathbb{R}^{2n}\times S^1,\xi_{st})$ the map $\Lambda_\phi$ does not need to be a section anymore, however it is smoothly isotopic to the $0$-section of $J^1(S^{2n}\times S^1)$ through a path of Legendrians thanks to Lemma \ref{lemme1}. We will describe in the next paragraph how one can associate to such a Legendrian a function $f: S^{2n}\times S^1\times\mathbb{R}^N\to\mathbb{R}$, for some $N\in\mathbb{N}$, such that we have again a one-to-one correspondence between critical points of $f$ of critical value $t$ and $t$-translated points of $\phi$. Moreover a control of the behaviour of $f$ at infinity will allow again to ensure existence of critical points of such a function $f$ and so the existence of translated points of $\phi$. 

\subsection{Generating functions}\label{generating function}

Let $X$ be a smooth manifold. For any integer $N\in\mathbb{N}$, a function
\[f: X\times\mathbb{R}^N\to\mathbb{R}\ \ \ \ \ \  (x,v)\mapsto f(x,v)\]
is said to be a generating function if $0$ is a regular value of
\[\frac{\partial f}{\partial v} : X\times\mathbb{R}^N\to(\mathbb{R}^N)^*\]
where $(\mathbb{R}^N)^*$ is the set of linear forms on $\mathbb{R}^N$, and where $\frac{\partial f}{\partial v}$ is the derivative of $f$ in the $\mathbb{R}^N$ direction. It follows from the definition that $\Sigma_f:=\left(\frac{\partial f}{\partial v}\right)^{-1}\{0\}$ is a smooth submanifold of $X\times\mathbb{R}^N$ of the same dimension of $X$ whenever $f : X\times\mathbb{R}^N\to\mathbb{R}$ is a generating function. In this case, denoting by $\frac{\partial f}{\partial x}$ the derivative of $f$ in the $X$ direction, the map \[\begin{aligned}
j^1_f :\Sigma_f&\to J^1X\\
(x,v)&\mapsto \left(x,\frac{\partial f}{\partial x}(x,v),f(x,v)\right)
\end{aligned}\]
is a Legendrian immersion. We say that the immersed Legendrian $\Lambda_f:=j^1_f(\Sigma_f)$ is generated by $f$. For all $a\in\mathbb{R}$ there is a one-to-one correspondence between critical points of $f$ with critical value $a$ and intersections of the immersed Legendrian $\Lambda_f$ with $\mathbb{O}_X\times\{a\}$.\\

\noindent
When $X$ is a compact manifold, a sufficient condition to guarantee the existence of critical points of a smooth function $f :X\times\mathbb{R}^N\to\mathbb{R}$  is to ask that the function $f$ is quadratic at infinity, i.e.\ there exists a non-degenerate quadratic form $Q :\mathbb{R}^N\to\mathbb{R}$ and a compactly supported function $g:X\times\mathbb{R}^N\to\mathbb{R}$ such that $f(x,v)=g(x,v)+Q(v), \text{ for all } (x,v)\in X\times\mathbb{R}^N$. The fundamental result about existence of generating functions is the following.\\

\begin{theoreme}[Chaperon \cite{Chaperon}, Chekanov \cite{chekanov}]\label{existence}
Let $X$ be a compact manifold and $\{\Lambda^t\}_{t\in[0,1]}$ a smooth path of Legendrians in $(J^1X,\xi^X)$ such that $\Lambda^0=\mathbb{O}_X\times\{0\}$. Then there exists an integer $N\in\mathbb{N}$ and a continuous path $\left\{f^t :X\times\mathbb{R}^N\to\mathbb{R}\right\}$ of generating functions quadratic at infinity such that $f^t$ generates the Legendrian $\Lambda^t$ for all $t\in[0,1]$. \\
\end{theoreme}

\subsection{Extracting critical values of generating functions}\label{minimax}

Classical minimax methods can be applied to extract a critical value of a function $f:X\times\mathbb{R}^N\to\mathbb{R}$ that is quadratic at infinity. Indeed, for any $a\in\mathbb{R}$ we write $X^a:=\{f\leq a\}$ to designate the sublevel $a$ of $f$. Since $f$ is quadratic at infinity, there exists $a_0\in\mathbb{R}$ such that for all $a\leq a_0$ the homotopy type of the sublevels $X^a$ and $X^{a_0}$ are the same, and we write $X^{-\infty}$ to refer to such sublevel sets. In addition to that there exists a splitting of $\mathbb{R}^N=\mathbb{R}^{N_+}\times\mathbb{R}^{N_-}$ for which the non-degenerate quadratic form $Q$ is negative definite on $\mathbb{R}^{N_-}$, and so the Thom isomorphism guarantees the existence of an isomorphism 
\[ T : H^*(X)\to H^{*+{N_-}}(X\times\mathbb{R}^N,X^{-\infty}),\]
where $H^*(X)$ is the cohomology of $X$ with coefficient in $\mathbb{Z}_2$ and $H^*(X\times\mathbb{R}^N,X^{-\infty})$ is the relative cohomology of $(X\times\mathbb{R}^N,X^{-\infty})$ with coefficient in $\mathbb{Z}_2$. For any $u\in H^*(X)\setminus\{0\}$ the number
\[c(f,u):=\inf\left\{a\in\mathbb{R}\ |\ i_a^*(T(u))\ne 0\right\},\]
where $i_a : (X^a,X^{-\infty})\hookrightarrow (X\times\mathbb{R}^N,X^{-\infty})$ is the inclusion, is a critical value of $f$.\\

\begin{remarque}\label{remarque minimax}
Zapolski in \cite{Zapolsky} uses homology instead of cohomology to extract the wanted critical values. Let $f:X\times\mathbb{R}^N\to\mathbb{R}$ be a generating function quadratic at infinity and $\alpha\in H_*(X)\setminus\{0\}$ a non zero homology class. Then
\[\mathcal{C}(f,\alpha):=\inf\{a\in\mathbb{R}\ |\ \widetilde{T}\alpha\in (i_a)_*(H_{*+{N_-}}(X^a,X^{-\infty}))\}, \]
where  $\widetilde{T} : H_*(X)\to H_{*+N_-}(X\times\mathbb{R}^N,X^{-\infty})$ comes from the Thom isomorphism in homology, is a critical value of $f$. If $\alpha_0$ is a generator of $H_n(X)$, and $\mu_0$ a generator of $H^n(X)$, where $n\in\mathbb{N}$ is the dimension of $X$, then $\mathcal{C}(f,\alpha_0)=c(f,\mu_0)$ (see for instance \cite{Viterbo}).\\
\end{remarque}

\noindent
Due to the uniqueness property of generating functions quadratic at infinity proved by Viterbo \cite{Viterbo} and Théret \cite{theret}, one deduces that if $f_1$ and $f_2$ are two generating functions quadratic at infinity that generate the same Legendrian $\Lambda\subset J^1X$ that is Legendrian isotopic to $\mathbb{O}_X\times\{0\}$, then $c(f_1,u)=c(f_2,u)$ for any $u\in H^*(X)\setminus\{0\}$. Thanks to this and to Theorem \ref{existence}, for any Legendrian $\Lambda$ that is Legendrian isotopic to $\mathbb{O}_X\times\{0\}$  and for any $u\in H^*(X)\setminus\{0\}$ the number $c(\Lambda,u):=c(f,u)$ is well defined, i.e.\ it does not depend on the generating function quadratic at infinity $f$ that generates $\Lambda$. Combining this with the previous discussion in Subsection \ref{generating function} we deduce that
\begin{equation}\label{eq1}
\Lambda\cap \mathbb{O}_X\times\{c(\Lambda,u)\}\ne\emptyset.
\end{equation}

\noindent
The translation selector constructed by Sandon is then given by 
\[c :\text{Cont}_0^c(\mathbb{R}^{2n}\times S^1,\xi_{st})\to\mathbb{R}\ \ \ \ \ \ \phi\mapsto c(\Lambda_\phi,u_0),\]
where $u_0$ is a generator of $H^{2n+1}(S^{2n}\times S^1)$ and $\Lambda_\phi\subset J^1(S^{2n}\times S^1)$ is the Legendrian associated to $\phi$ constructed in Subsection \ref{section graphe}. By \eqref{eq1} and Lemma \ref{lemme1} we come to the conclusion that $c(\phi)\in\text{Spectrum}(\phi)$. Furthermore Sandon in \cite{sandonthese} proved the following result.

\begin{theoreme}[\cite{sandonthese}]\label{translation selector}
The translation selector $c$ satisfies the following properties:
\begin{enumerate}
    \item $c(\phi)\geq 0$ for any $\phi$;
    \item if $\{\phi^t\}$ is a smooth path of compactly supported contactomorphisms starting at the identity, then $t\mapsto c(\phi^t)$ is continuous;
    \item $\left\lceil c(\varphi)\right\rceil \geq\left\lceil c(\phi\varphi)-c(\phi)\right\rceil$, in particular $\left\lceil c(\phi\varphi)\right\rceil\leq\left\lceil c(\phi)\right\rceil+\left\lceil c(\varphi)\right\rceil$ for any $\phi,\varphi$;
    \item if $\{\phi^t\}$ and $\{\varphi^t\}$ are smooth paths of compactly supported contactomorphisms starting at the identity such that $\alpha_{st}\left(\frac{d}{dt}\phi^t(x)\right)\leq\alpha_{st}\left(\frac{d}{dt}\varphi^t(x)\right)$ for all $t\in[0,1]$ and $x\in\mathbb{R}^{2n}\times S^1$ then $c(\phi^1)\leq c(\varphi^1)$.\\
\end{enumerate}
\end{theoreme}

\begin{remarque}\label{remarque sur ordonnabilite}
   In \cite{sandonthese} Sandon shows that the translation selector satisfies also the following properties 
   \begin{enumerate}
       \item $c(\phi)=c(\phi^{-1})=0$ if and only if $\phi=\Id$. 
       \item $\left\lceil c(\varphi\phi\varphi^{-1})\right\rceil=\left\lceil c(\phi)\right\rceil$.
   \end{enumerate}
   These extra properties of the translation selector will not be used for the proof of Theorem \ref{geodesique shelukhin/fpr} and Theorem \ref{geodesique discriminante/oscillation}. However let us notice that they are used in \cite{sandonthese}, \cite{sandonmetrique} to deduce (universal) orderability of $(\mathbb{R}^{2n}\times S^1,\xi_{st})$ (similarly to the case of $(\mathbb{R}^{2n+1},\xi_{st})$ done by Bhupal \cite{bhupal}) and to show that the map
\[ \nu_c :\Cont_0^c(\mathbb{R}^{2n}\times S^1,\xi_{st})\to\mathbb{N}\ \ \ \ \phi\mapsto \left\lceil c(\phi)\right\rceil +\left\lceil c(\phi^{-1})\right\rceil \] 
is a conjugation invariant norm. 

\end{remarque}

\section{Proof of the results}\label{section proof}

The idea of the proof of Theorem \ref{geodesique shelukhin/fpr} and Theorem \ref{geodesique discriminante/oscillation} is the following. First we will compute the selected translation of contactomorphisms that are generated by Hamiltonian functions satisfying the hypotheses of Theorems \ref{geodesique shelukhin/fpr} and \ref{geodesique discriminante/oscillation}. Then we will show that all the norms we are working with are bounded from below by the translation selector. Finally a direct computation will allow to show that the length of the paths we are considering are equal to the selected translation of the time-one of these paths. We thus deduce that these paths are length minimizing paths and so are geodesics.\\

\subsection{Computation of the selected translation}

To any compactly supported time dependent function $H:[0,1]\times\mathbb{R}^{2n}\to\mathbb{R}, (t,p)\mapsto H^t(p)$, one can associate the smooth path of vector fields $ X^{{\omega_{st}}}_H : [0,1]\to\chi(\mathbb{R}^{2n})$ defined by $\iota_{X^{{\omega_{st}}}_H(t)}{\omega_{st}}=-dH^t$ for all $t\in\mathbb{R}$, where ${\omega_{st}}=\sum\limits_{i=1}^ndx_i\wedge dy_i$. We denote by $\psi_H^t$ its time $t$-flow for any $t\in\mathbb{R}$. Note that $\psi_H^t$ is an Hamiltonian symplectomorphism for any $t\in\mathbb{R}$. Recall that from the formula \eqref{contact vector field} of the introduction one can similarly associate to a smooth compactly supported function $h:[0,1]\times\mathbb{R}^{2n}\times S^1\to\mathbb{R}$ a path of contactomorphisms $\phi_h$.\\

In the next lemma for all $z\in\mathbb{R}$ we will write $[z]$ the corresponding element in $\mathbb{R}/\mathbb{Z}=S^1$ and debote by $\lambda_{st}$ the $1$-form $\sum\limits_{i=1}^ny_idx_i$ on $\mathbb{R}^{2n}$. \\

\begin{lemme}\label{relevé de contact}
Let $\{\phi_h^t\}$ be a smooth path of contactomorphisms generated by the Hamiltonian function 
\[\begin{aligned}
h:\mathbb{R}^{2n}\times S^1&\to\mathbb{R}\\
(p,z) &\mapsto H(p)
\end{aligned}\]
where $H : \mathbb{R}^{2n}\to\mathbb{R}$ is a smooth compactly supported function. Then for all $(p,z)\in\mathbb{R}^{2n}\times \mathbb{R}$ and for all $t\in\mathbb{R}$ 
\[\phi_h^t(p,[z])=\left(\psi_H^t(p),[z+F^t(p)]\right)\ \ \text{ and }\ \  (\phi_h^t)^*\alpha_{st}=\alpha_{st}\]
where
\begin{equation}\label{2} F^t(p)=\int_0^t\lambda_{st}\left(\psi_H^s(p)\right)\left(X^{\omega_{st}}_H\left(\psi_H^s(p)\right)\right)ds+tH(p).
\end{equation}
\end{lemme}

\vspace{2.ex}

\begin{proof}
A direct computation shows that the contact vector field $X_h$ generated by the function $h:\mathbb{R}^{2n}\times S^1\to\mathbb{R}$ (see the relations \eqref{contact vector field} in the introduction) in this case is given by 
\[X_h(p,[z])=X^{\omega_{st}}_H(p)+\left(H(p)+\lambda_{st}(p)\left(X^{\omega_{st}}_H(p)\right)\right)\frac{\partial}{\partial z}\ ,\ \text{ for all } (p,[z])\in\mathbb{R}^{2n}\times S^1.\]
Moreover, because the Hamiltonian function $H$ is time-independent, $H$ is constant along its flow, i.e.\ $H(\psi_H^t(p))=H(p)$ for all $p\in\mathbb{R}^{2n}$ and for all $t\in\mathbb{R}$. So we deduce the formula \eqref{2}.  To prove the fact that $\phi_h^t$ is an exact contactomorphism for all $t$, i.e.\ it preserves not only the contact distribution $\xi_{st}$ but also the contact form $\alpha_{st}=dz-\lambda_{st}$, we use the Cartan formula :
\[\left.\frac{d}{dt}\right\lvert_{t=s}(\phi_h^t)^*\alpha_{st}=(\phi_h^s)^*\left(\iota_{X_h}d\alpha_{st}+d\iota_{X_h}\alpha_{st}\right)=(\phi_h^s)^*\left(-dH+dH\right)\equiv 0\ ,\ \forall s\in\mathbb{R}\ .\]
Since $\phi_h^0=\Id$, we deduce that $\left(\phi_h^s\right)^*\alpha_{st}=\left(\phi_h^0\right)^*\alpha_{st}=\alpha_{st}$ for all $s\in\mathbb{R}$.\end{proof}

\vspace{2.ex}

\noindent
The hypothesis that we make about the smallness of the Hessian of $H:\mathbb{R}^{2n}\to\mathbb{R}$ allows us to guarantee that the only periodic orbits of $\{\psi_H^t\}_{t\in\mathbb{R}}$ of small periods are the constant ones. More precisely, identifying in the usual way the dual Euclidean space $(\mathbb{R}^{2n})^*$ with $\mathbb{R}^{2n}$, one can see the Hessian of any smooth function $H:\mathbb{R}^{2n}\to\mathbb{R}$ at a point $p\in\mathbb{R}^{2n}$ as a linear map $\Hess_p(H):\mathbb{R}^{2n}\to\mathbb{R}^{2n}$. The smallness of such maps will be measured in terms of the operator norm, i.e.\ if $A:\mathbb{R}^{2n}\to\mathbb{R}^{2n}$ is a linear map then \[\left|\left|\left|A\right|\right|\right|:=\underset{v\in\mathbb{R}^{2n}\setminus\{0\}}\sup \frac{|A v|}{|v|} \text{ where } |v| \text{ is the standard Euclidean norm of } v\in\mathbb{R}^{2n}\setminus\{0\} .\]

\begin{lemme}\label{petite hessienne}
Let $H:\mathbb{R}^{2n}\to\mathbb{R}$ be a compactly supported Hamiltonian function such that $\underset{p\in\mathbb{R}^{2n}}\sup|||\operatorname{Hess}_p(H)|||<2\pi$. If there exists $0<T\leq 1$ and $p\in\mathbb{R}^{2n}$ such that $\psi_H^T(p)=p$ then $\psi_H^t(p)=p$ for every $t\in\mathbb{R}$.\\ 
\end{lemme}

\begin{proof}
Let $T\in\mathbb{R}_{>0}$ and $p\in\mathbb{R}^{2n}$ be such that $\psi_H^T(p)=p$ and denote by $\gamma :\mathbb{R}/\mathbb{Z}\to\mathbb{R}^{2n}$ the loop $\gamma(t)=\psi_H^{tT}(p)$. The speed $t\mapsto \dot\gamma(t)=X_{TH}(\gamma(t))$ of $\gamma$ is again a smooth loop. Consider the Fourier series associated to this loop $\dot\gamma$
\[\sum\limits_{k\in\mathbb{Z}}e^{2i\pi Jt}x_k \text{ for all } t\in S^1 \text{ where } x_k\in\mathbb{R}^{2n} \text{ and } J=\left (
   \begin{array}{c c c}
      0 & \Id_n \\
      \text{-Id}_n & 0\\
   \end{array}\ 
\right)\ .\]
Note that $\dot\gamma(t)=JT\nabla H(\gamma(t))$ for any $t\in S^1$, where $\nabla H$ is the Euclidean gradient of $H$, so $\ddot\gamma(t)=T\operatorname{Hess}_{\gamma(t)}H(\dot\gamma(t))$ for any $t\in S^1$ and the Fourier series of the loop $\ddot\gamma$ is given by \[t\mapsto \sum\limits_{k\in\mathbb{Z}}2i\pi Je^{2i\pi Jt}x_k\ .\] Using the fact that $\int_0^1\dot\gamma(t)dt=\gamma(1)-\gamma(0)=x_0=0\in\mathbb{R}^{2n}$ and the Parseval's identity we have 
\begin{equation}\label{Fourier}
||\dot\gamma||_{\mathcal{L}^2}=\sqrt{\sum\limits_{k\in\mathbb{Z}\setminus\{0\}}|x_k|^2}\leq\sqrt{\sum\limits_{k\in\mathbb{Z}\setminus\{0\}} k^2|x_k|^2}=\frac{1}{2\pi}||\ddot\gamma||_{\mathcal{L}^2}\ .
\end{equation}
Moreover the above inequality is a strict inequality in the case when $\gamma$ is not the constant loop.\\
On the other hand
\[\begin{aligned}||t\mapsto\ddot\gamma(t)||_{\mathcal{L}^2}&=||t\mapsto T\Hess_{\gamma(t)}H(\dot\gamma(t))||_{\mathcal{L}^2}\\
&=T||t\mapsto\sum\limits_{k\in\mathbb{Z}}e^{2i\pi Jt}\Hess_{\gamma(t)} H(x_k)||_{\mathcal{L}^2}\\
&\leq T\sqrt{\sum\limits_{k\in\mathbb{Z}}\underset{p\in\mathbb{R}^{2n}}\sup |||\Hess_p(H)|||^2|x_k|^2}\\
&=T\underset{p\in\mathbb{R}^{2n}}\sup|||\Hess_p(H)|||\sqrt{\sum\limits_{k\in\mathbb{Z}\setminus\{0\}}|x_k|^2}\\
&\leq T2\pi||\dot\gamma||_{\mathcal{L}^2}\ .
\end{aligned}\]
So combining this inequality with the inequality \eqref{Fourier} we deduce that $||\dot\gamma||_{\mathcal{L}^2}<T||\dot\gamma||_{\mathcal{L}^2}$ when $\gamma$ is not a constant loop, and so $T>1$.
\end{proof}

\begin{remarque}
  The reader can find a similar proof in \cite{hoferzehnder}.
\end{remarque}

\noindent
From the two previous lemmas we deduce the selected translation of the paths considered in Theorems \ref{geodesique shelukhin/fpr} and \ref{geodesique discriminante/oscillation} .\\

\begin{proposition} \label{calcul explicite}
Let $\{\phi^t\}_{t\in[0,1]}$ be a smooth path of contactomorphisms generated by a compactly supported Hamiltonian function 
\[h : \mathbb{R}^{2n}\times S^1\to\mathbb{R}\ \ \ \ \ \ 
(p,[z])\mapsto H(p)\]
with $\underset{p\in\mathbb{R}^{2n}}\sup |||\Hess_p(H)|||<2\pi$. Then
\[c(\phi_h^1)=\max H\ \  \text{ and }\ \  c\left((\phi_h^1)^{-1}\right)=-\min H.\]
\end{proposition}
\vspace{2.ex}

\begin{proof}
First let us show that \[\text{Spectrum}(\phi_h^t)=\{tH(p)\ |\ d_pH=0\} \text{ for all } t\in[0,1].\] 
Indeed if $p\in\mathbb{R}^{2n}$ is such that $d_pH=0$ then using Lemma \ref{relevé de contact} we have that for all $[z]\in S^1$ and $t\in[0,1]$
\[\phi_h^t(p,[z])=\left(p,[z+tH(p)]\right)\ \ \ \text{ and }\ \ \  (\phi_h^t)^*\alpha_{st}=\alpha_{st}\ ,\]
and so $\{tH(p)\ |\ d_pH=0\}\subset\text{Spectrum}(\phi_h^t)$. Conversely, let $t\in[0,1]$ and $a\in\text{Spectrum}(\phi_h^t)$. By definition of the spectrum and using again Lemma \ref{relevé de contact}, this implies that there exists $(p,z)\in\mathbb{R}^{2n}\times \mathbb{R}$ such that 
\[\widetilde{\phi_h^t}(p,z)=(p,z+a)=\left(\psi_H^t(p),z+F^t(p)\right)\ .\]
In particular $p$ is a fixed point of $\psi_H^t$. Lemma \ref{petite hessienne} ensures that under the assumptions we made on the Hessian of $H$ we have $\psi_H^s(p)=p$ for all $s\in[0,1]$ and $d_pH=0$. So we conclude that $F^t(p)=tH(p)=a$ and that $\text{Spectrum}(\phi_h^t)\subset\{tH(p)\ |\ d_pH=0\}$. \\

\noindent
Moreover, by Theorem \ref{translation selector}, the map $t\mapsto c(\phi_h^t)\in\text{Spectrum}(\phi_h^t)$ is continuous, and since the set of critical values of $H$ is a nowhere dense set due to Sard's theorem, there exists a critical point $p_1\in\mathbb{R}^{2n}$ of $H$ such that $c(\phi_h^t)=tH(p_1)$ for all $t\in [0,1]$. It remains to show that $H(p_1)=\max H$. \\

\noindent
For $\varepsilon>0$ small enough, $\Lambda_{\phi_h^t} : S^{2n}\times S^1\to J^1(S^{2n}\times S^1)$ is a Legendrian section for all $t\in[0,\varepsilon]$, and so $\Lambda_{\phi_h^t}=j^1f^t$ is equal to the $1$-jet of a generating function $f^t : S^{2n}\times S^1\to\mathbb{R}$ without extra coordinates. It is well known that in this case $c(f^t,u_0)=\max f^t$ when $u_0$ denotes the generator of $H^{2n+1}(S^{2n}\times S^1)$: to see this one can combine Remark \ref{remarque minimax} with the arguments of Chapter 10 of \cite{polterovichrosen}. In addition to this, since $f^t$ is a generating function for $\phi_h^t$ we have
\[\{\text{Critical values of } f^t\}=\text{Spectrum}\left(\phi_h^t\right)=\{tH(p)\ |\ d_pH=0\}\ .\]
We then deduce that \[c(\phi_h^t)=\max\ f^t=t\max H=tH(p_1),\]
so $H$ reaches its maximum at the point $p_1$. Thus $c(\phi_h^1)=\max H$. Noticing that the Hamiltonian function $-h$ generates the path $\{(\phi_h^t)^{-1}\}=\{\phi_h^{-t}\}$ the same proof allows us to show that $c((\phi_h^1)^{-1})=-\min H$.\end{proof}

\vspace{2.ex}

\noindent
\subsection{A lower bound for the Shelukhin norms and the FPR norms}

The following proposition will allow us to find a lower bound for the Shelukhin and FPR norms in terms of the translation selector.\\

\begin{proposition}\label{inégalité géodésique shelukhin}
Let $\phi\in\Cont_0^c(\mathbb{R}^{2n}\times S^1,\xi_{st})$ be a contactomorphism and $k:[0,1]\times \mathbb{R}^{2n}\times S^1\to\mathbb{R}$ a compactly supported Hamiltonian that generates $\phi$, i.e.\ $\phi_k^1=\phi$. Then
\[c(\phi)\leq\int_0^1\max k^t dt\ \ \ \ \text{ and }\ \ \ \ c(\phi^{-1})\leq -\int_0^1\min k^tdt\ .\]
\end{proposition}

\vspace{2.ex}
\noindent
To prove this proposition we will use Lemma 2.6 of \cite{Zapolsky}. To state his lemma Zapolski in \cite{Zapolsky} fixes a non-zero homology class. We will state this lemma by fixing a non-zero cohomology class instead. By Remark \ref{remarque minimax} the two approaches are equivalent. \\

\begin{lemme}[\cite{Zapolsky}]\label{zapolski}
Let $X$ be a closed manifold of dimension $n$, $\{\Phi^t\}_{t\in[0,1]}$ be a smooth path of contactomorphisms in $\Cont_0^c(J^1X,\Ker(\alpha_X))$ starting at the identity and $u_0\in H^n(X)\setminus\{0\}$ be the top class. Then 
\[ c(\Phi^1(\mathbb{O}_X\times\{0\}),u_0)\leq \int_0^1\underset{x\in X}\max\ \alpha_X\left(\left.\frac{d}{ds}\right\lvert_{s=t}\Phi^s(\sigma_0(x))\right)dt\ ,\]
where $\sigma_0 : X\to J^1X$, $x\mapsto (x,0,0)$ is the $0$-section.\\
\end{lemme}

\noindent
In the following proof of Proposition \ref{inégalité géodésique shelukhin} we use the notations and construction of Subsection \ref{section graphe}.\\

\begin{proof}[Proof of Proposition \ref{inégalité géodésique shelukhin}]
Let $\widetilde{\phi_k}^t\in\Cont_0^c(\mathbb{R}^{2n+1},\xi_{st})^{z-per}$ be the $1$-periodic contactomorphism associated to $\phi_k^t$ for all $t\in[0,1]$ and
\[ [0,1]\times S^{2n}\times S^1\to J^1(S^{2n}\times S^1)\ \ \ \ \ \ \ (t,(p,[z]))\mapsto \Lambda_{{\phi_k}^t}(p,[z])\]
 the associated Legendrian isotopy. The Legendrian isotopy extension theorem guarantees the existence of a smooth path of contactomorphisms $\{\Phi^t\}\subset\Cont_0^c\left(J^1(S^{2n}\times S^1),\Ker(\alpha_{S^{2n}\times S^1})\right)$ starting at the identity such that $\Phi^t\circ\Lambda_{\Id}\equiv \Lambda_{{\phi_k}^t}$ for all $t\in[0,1]$. We claim that for all $(p,[z])\in S^{2n}\setminus\{p_0\}\times S^1$ and for all $t\in[0,1]$
\begin{equation}\label{equation4}
\alpha_{S^{2n}\times S^1}\left(\frac{d}{dt}\Phi^t(\Lambda_{\Id}(p,[z]))\right)=k^t(\phi_k^t(\overline{\psi}(p,[z])))\ .
\end{equation}
Indeed, since the maps $\Psi$, $\text{pr}$ and $\Theta$ of Subsection \ref{section graphe} preserve the contact forms we have 
\[\begin{aligned}
\alpha_{S^{2n}\times S^1}\left(\frac{d}{dt}\Phi^t(\Lambda_{\Id}(p,[z]))\right)&=\alpha_{S^{2n}\times S^1}\left(\frac{d}{dt}\Psi^{-1}\left(\Lambda^{\mathbb{R}^{2n}\times S^1}_{\phi_k^t}(\overline{\psi}(p,[z]))\right)\right)\\
&=(\Psi^{-1})^*\alpha_{S^{2n}\times S^1}\left(\frac{d}{dt}\text{pr}\left(\Lambda^{\mathbb{R}^{2n+1}}_{\widetilde{\phi_k}^t}\left(\overline{\psi}(p,z)\right)\right)\right)\\
&=\text{pr}^*\alpha_{\mathbb{R}^{2n}\times S^1}\left(\frac{d}{dt}\Theta\left(\text{gr}_{\alpha_{st}}(\widetilde{\phi_k^t})\left(\overline{\psi}(p,z)\right)\right)\right)\\
&=\Theta^*\alpha_{\mathbb{R}^{2n+1}}\left(\frac{d}{dt}\text{gr}_{\alpha_{st}}(\widetilde{\phi_k^t})(\overline{\psi}(p,z))\right)\\
&=(\alpha^2_{st}-e^\theta\alpha^1_{st})\left(\frac{d}{dt}\left((p,z),\widetilde{\phi_k^t}(p,z),\tilde{g}_k^t(\overline{\psi}(p,z))\right)\right)\\
&=\alpha_{st}^2\left(\frac{d}{dt}\widetilde{\phi_k^t}(p,z)\right)-e^\theta\alpha_{st}^1\left(0,\frac{d}{dt}\widetilde{g_k^t}(\overline{\psi}(p,z))\right)\\
&=k^t(\phi_k^t(\overline{\psi}(p,[z]))).
\end{aligned}\]
Moreover for all $[z]\in S^1$ and for all $t\in[0,1]$ a direct computation shows that \begin{equation}\label{equation5}
\alpha_{S^{2n}\times S^1}\left(\frac{d}{dt}\Phi^t(\Lambda_{\Id}(p_0,[z]))\right)=0\ ,
\end{equation}
since $\Lambda_{\phi_k^t}(p_0,[z])=((p_0,[z]),0)\subset J^1(S^{2n}\times S^1)$ is independent of time. So by Lemma \ref{zapolski} and equalities \eqref{equation4}, \eqref{equation5}  we have
\[\begin{aligned}
c(\phi_k^1)=c(\Lambda_{\phi_k^1},u_0)&\leq \int_0^1\underset{(p,z)\in S^{2n}\times S^1}\max\alpha_{S^{2n}\times S^1} \left(\frac{d}{dt}\Phi^t(\Lambda_{Id}(p,z))\right)dt\\
&=\int_0^1\max\left\{\underset{S^{2n}\setminus\{p_0\}\times S^1}\max k^t(\phi_k^t(\overline{\psi}(p,z))),0\right\}dt\\
&=\int_0^1\max k^tdt\ ,
\end{aligned}\]
which proves the first inequality of Proposition \ref{inégalité géodésique shelukhin}. Using the fact that $(\phi_k^1)^{-1}$ can be generated by the Hamiltonian function 
\[ [0,1]\times\mathbb{R}^{2n}\times S^1\to\mathbb{R}\ \ \ \ (t,p,z)\mapsto -k^{1-t}(p,z)\] we deduce the second inequality of Proposition \ref{inégalité géodésique shelukhin} exactly in the same way.\end{proof}

\vspace{2.ex}

\noindent
We deduce the following corollary.\\

\begin{corollaire}\label{corollaire shelukhin}
Let $k :[0,1]\times\mathbb{R}^{2n}\times S^1\to\mathbb{R}$ be a compactly supported Hamiltonian function. Then
\[\begin{aligned}
\max\left\{c(\phi_k^1),c\left((\phi_k^1)^{-1}\right)\right\}&\leq\nu_S\left(\phi_k^1\right)\leq\widetilde{\nu_S}\left([\phi_k]\right)\ \text{ and }\\
\max\left\{\left\lceil c(\phi_k^1)\right\rceil,\left\lceil c\left((\phi_k^1)^{-1}\right)\right\rceil\right\}&\leq \nu_{FPR}\left(\phi_k^1\right)\leq\widetilde{\nu_{FPR}}\left([\phi_k]\right).
\end{aligned}\]
\end{corollaire}
\vspace{2.ex}

\begin{proof}
By Proposition \ref{inégalité géodésique shelukhin},
\[\max\left\{c(\phi_k^1),c\left((\phi_k^1)^{-1}\right)\right\}\leq\max\left\{\int_0^1\max k^tdt,-\int_0^1\min k^tdt\right\}.\]
By definition of the Shelukhin length functional we have
\[\max\left\{\int_0^1\max k^tdt,-\int_0^1\min k^tdt\right\}\leq \mathcal{L}_S^{\alpha_{st}}(\phi_k)\]
for any compactly supported Hamiltonian function $k :[0,1]\times\mathbb{R}^{2n}\times S^1\to\mathbb{R}$. By taking the infinimum of the length over all paths that represents $\phi_k^1$ (resp.\ $[\{\phi_k\}]$) we deduce the first line of inequalities of Corollary \ref{corollaire shelukhin}. The second line of inequalities can be proved exactly in the same way and is left to the reader.\end{proof}

\vspace{2.ex}

\noindent

\subsection{Computation of the Shelukhin length and proof of Theorem \ref{geodesique shelukhin/fpr}}
The Shelukhin length of a path of contactmorphisms $\{\phi^t\}\subset\Cont_0^c(\mathbb{R}^{2n}\times S^1,\xi_{st})$ generated by a time independant compactly supported Hamiltonian function $h:\mathbb{R}^{2n}\times S^1\to\mathbb{R}$ is given by $\mathcal{L}_S^{\alpha_{st}}(\{\phi^t\})=\max\{\max h,-\min h\}$. Using Corollary \ref{corollaire shelukhin} we then deduce that the inequalities of Proposition \ref{inégalité géodésique shelukhin} are equalities in the case when $\{\phi^t\}$ is generated by an Hamiltonian function satisfying the hypothese of Theorem \ref{geodesique shelukhin/fpr} and so that it is a geodesic of the Shelukhin norm. A similar argument allows to compute the FPR norm of $\{\phi^t\}$ in this case and conclude the proof of Theorem \ref{geodesique shelukhin/fpr}.

\begin{proof}[Proof of Theorem \ref{geodesique shelukhin/fpr}]

Thanks to Corollary \ref{corollaire shelukhin} and Lemma \ref{calcul explicite} we deduce that
\[\max\left\{\max h,-\min h\right\}\leq \nu_S(\phi_h^1)\leq\widetilde{\nu_S}([\phi_h])\ .\]
On the other side, by definition we have \[\nu_S\left(\phi_h^1\right)\leq\widetilde{\nu_S}\left([\phi_h]\right)\leq\mathcal{L}_S^{\alpha_{st}}\left(\phi_h\right)=\int_0^1\underset{(t,(p,z))}\max|h^t(p,z)|dt=\max\{\max h,-\min h\} .\]
So all these inequalities are in fact equalities 
\[\mathcal{L}_S^{\alpha_{st}}\left(\phi_h\right)=\widetilde{\nu_S^{\alpha_{st}}}\left([\phi_h]\right)=\nu_S^{\alpha_{st}}\left(\phi_h^1\right)=\max\{\max h,-\min h\}\ .\]
In the same way, using Corollary \ref{corollaire shelukhin} and Lemma \ref{calcul explicite} we have
\[\max\left\{\left\lceil \max h\right\rceil, \left\lceil -\min h\right\rceil\right\}\leq\nu_{FPR}^{\alpha_{st}}\left(\phi_h^1\right)\leq\widetilde{\nu_{FPR}^{\alpha_{st}}}\left(\phi_h\right)\ ,\]
and by definition of the FPR norm we have
\[\nu^{\alpha_{st}}_{FPR}\left(\phi_h^1\right)\leq\widetilde{\nu^{\alpha_{st}}_{FPR}}\left([\phi_h]\right)\leq \max\{\left\lceil\max h\right\rceil,\left\lceil-\min h\right\rceil\}\ .\]
Again we deduce that all these inequalities are in fact equalities. Finally using the fact the the FPR norm is conjugation invariant we get the desired result.
\end{proof}

\vspace{2.ex}

The idea to prove the results for the discriminant and oscillation norm will follow the same lines: we have to show that the translation selector is a lower bound for these norms, and that the discriminant/oscillation length of the considered path realizes this lower bound. 

\subsection{A lower bound for the discriminant and oscillation norms}

In the next proposition we formulate precisely how the translation selector gives a lower bound for the discriminant and oscillation norms.\\

\begin{proposition}\label{crucial}
For any element $\phi\in\Cont_0^c(\mathbb{R}^{2n}\times S^1,\xi_{st})$ that is not the identity we have
\[\begin{aligned}
\max\left\{\left\lfloor c(\phi)\right\rfloor+1,\left\lfloor c\left(\phi^{-1}\right)\right\rfloor+1\right\}&\leq \nu_d(\phi) \text{ and }\\
\max\left\{\left\lfloor c(\phi)\right\rfloor+1,\left\lfloor c\left(\phi^{-1}\right)\right\rfloor+1\right\}&\leq \nu_{osc}(\phi).
\end{aligned}\]
\end{proposition}

\begin{proof}

\vspace{2.ex}

Let us first show that if $\{\phi^t\}$ is a smooth path of compactly supported contactomorphisms starting at the identity in $\Cont_0^c(\mathbb{R}^{2n}\times S^1,\xi_{st})$ such that $\mathcal{L}_d\left(\{\phi^t\}\right)=1$ then $0\leq c(\phi^t)<1$ for all $t\in [0,1]$. Indeed, since the application $t\in[0,1]\mapsto c(\phi^t)$ is continuous and $c(\phi^0)=0$, if there exists $t\in]0,1]$ such that $c(\phi^t)\geq 1$ then there exists $t_0\in]0,t]$ such that $c(\phi^{t_0})=1$. This means that $\phi^{t_0}$ has a $1$-translated point hence a discriminant point, which contradicts the fact that $\mathcal{L}_d\left(\{\phi^t\}\right)=1$. So $c(\phi^t)<1$ for all $t\in[0,1]$ whenever $\mathcal{L}_d\left(\{\phi^t\}\right)=1$ and $\phi^0=\Id$. \\

\noindent
Let us consider an element $\phi\in\Cont_0^c(\mathbb{R}^{2n}\times S^1,\xi_{st})\setminus\{\Id\}$. Denote by $k=\nu_d(\phi)\in\mathbb{N}_{>0}$ its discriminant norm. Using the word metric definition of the discriminant norm, this means that there exist $k$ smooth paths of compactly supported contactomorphisms $\left(\{\phi_i^t\}_{t\in[0,1]}\right)_{i\in[1,k]\cap\mathbb{N}}$ such that 
\begin{enumerate}
    \item each of them starts at the identity and is of discriminant length $1$
    \item $\phi=\prod\limits_{i=1}^k\phi_i^1$.
\end{enumerate}
The triangular inequality property that is satisfied by the translation selector (see Theorem \ref{translation selector}) and the previous estimate of the translation selector on paths of discriminant length equals to $1$ allow us to deduce that 
\[\left\lceil c(\phi)\right\rceil\leq\sum\limits_{i=1}^k \left\lceil c(\phi_i^1)\right\rceil\leq k=\nu_d(\phi)\ ,\]
and so that $\nu_d(\phi)\geq \left\lceil c(\phi)\right\rceil $. Since $\nu_d(\phi^{-1})=\nu_d(\phi)$ we deduce that 

\[\nu_d(\phi)\geq\max\left\{\left\lceil c(\phi)\right\rceil,\left\lceil c(\phi^{-1})\right\rceil\right\} .\]
So if $\max\{c(\phi),c(\phi^{-1})\}\notin\mathbb{N}$ we proved the desired first inequality \[\nu_d(\phi)\geq\max\{\left\lfloor c(\phi)\right\rfloor+1,\left\lfloor c(\phi^{-1})\right\rfloor+1 \}.\]

\noindent
In remains to show the case when $\max\{c(\phi),c(\phi^{-1})\}\in\mathbb{N}$. To set the ideas down, let us assume that $\max\{c(\phi),c(\phi^{-1})\}=c(\phi)$, the case where $\max\{c(\phi),c(\phi^{-1})\}=c(\phi^{-1})$ leads to the same conclusion by the same arguments.  We already proved that $k=\nu_d(\phi)\geq c(\phi)$, and we want to show that $\nu_d(\phi)=k>c(\phi)$. If $c(\phi)=1$ then using the argument of the first paragraph of this proof we know that $\nu_d(\phi)\geq 2>c(\phi)$. So it remains to show the inequality for the cases where $c(\phi)\geq 2$. Suppose by contradiction that $c(\phi)=c(\prod\limits_{i=1}^k\phi_i^1)=k$. Then $\lceil c(\phi_i^1)\rceil =1$ for all $i\in[1,k]$, in particular $c(\phi_i^1)\in ]0,1[$. Thanks to Theorem \ref{translation selector} we have $\lceil c(\prod\limits_{i=2}^{k}\phi_i^1)\rceil \geq \lceil c(\prod\limits_{i=1}^k\phi_i^1)-c(\phi_1^1)\rceil $. Since $c(\phi_1^1)\in]0,1[$ and $c(\Pi_{i=1}^k\phi_i^1)\in\mathbb{N}$ we deduce that
\[\lceil c(\Pi_{i=1}^k\phi_i^1)-c(\phi_1^1) \rceil > c(\Pi_{i=1}^k\phi_i^1)-\lceil c(\phi_1^1) \rceil\  \text{ so }\ c(\Pi_{i=1}^k\phi_i^1)<\lceil c(\phi_1^1) \rceil+\lceil c(\Pi_{i=2}^{k}\phi_i^1)\rceil\ .\] This leads us to the following contradiction 
\[k=c(\phi)=c(\Pi_{i=1}^k\phi_i^1)<\lceil c(\phi_1^1)\rceil +\lceil c(\overset{k}{\underset{i=2}{\Pi}}\phi_i^1)\rceil = \sum\limits_{i=1}^k\lceil c\left(\phi_i^1\right)\rceil = k\ ,\]
which concludes the proof of the first inequality of the proposition.\\

The proof for the second inequality concerning the oscillation norm goes in the same way but will use one more argument based on the fourth point of Theorem \ref{translation selector}. Let $\phi\in\Cont_0^c(\mathbb{R}^{2n}\times S^1,\xi_{st})\setminus\{\Id\}$. By definition there exists $[\{\phi^t\}]\in\widetilde{\Cont_0^c}(\mathbb{R}^{2n}\times S^1,\xi_{st})$ such that $\phi^1=\phi$ and  \[\nu_{osc}\left(\phi\right)=\widetilde{\nu_{osc}}\left([\{\phi^t\}]\right)=\max\left\{\widetilde{\nu_{+}^{osc}}\left([\{\phi^t\}]\right),-\widetilde{\nu_{-}^{osc}}\left([\{\phi^t\}]\right)\right\}\ .\]
To set the ideas down we will assume that $\nu_{osc}(\phi)=\widetilde{\nu_{osc}}\left([\{\phi^t\}]\right)=\widetilde{\nu_{+}^{osc}}\left([\{\phi^t\}]\right)$, the case where $\widetilde{\nu_{osc}}\left([\{\phi^t\}]\right)=-\widetilde{\nu_{-}^{osc}}\left([\{\phi^t\}]\right)$ leads to the same conclusion and is left to the reader.\\

\noindent
We will prove that $\widetilde{\nu_{+}^{osc}}\left([\{\phi^t\}]\right)\geq \left\lfloor c(\phi)\right\rfloor+1$. The proof that $-\widetilde{\nu_{-}^{osc}}\left([\{\phi^t\}]\right)\geq \left\lfloor c(\phi^{-1})\right\rfloor+1$ goes in the same way and is left to the reader.\\

\noindent
Let us denote by $k:=\widetilde{\nu_{+}^{osc}}\left([\{\phi^t\}]\right)$. By definition this means that there exists $N\in\mathbb{N}^*$ smooth paths $\left(\{\phi_i^t\}\right)_{i\in[1,N]\cap\mathbb{N}}$ such that 
\begin{enumerate}
    \item each of them is of discriminant length $1$, $k$ of them are positive and $N-k$ of them are negative
    \item $\phi=\prod\limits_{i=1}^N\phi_i^1$.
\end{enumerate}
If $\{\varphi^t\}$ is non-positive by Theorem \ref{translation selector} we have $c(\varphi^1)=0$. So using the triangular inequality we have again that $\left\lceil c(\phi)\right\rceil \leq k=\widetilde{\nu_+^{osc}}([\{\phi^t\}])$. \\

\noindent
If we assume that $c(\phi)\notin\mathbb{N}$ then $\nu_{osc}(\phi)\geq \left\lfloor c(\phi)\right\rfloor +1$ and we have the second inequality of the proposition. \\
 
 \noindent
Finally suppose that $c(\phi)\in\mathbb{N}$. If $c(\phi)=1$ the argument from the first paragraph of this proof allows us to show that $\nu_{osc}(\phi)\geq 2=\lfloor c(\phi)\rfloor+1$ which proves again the second inequality of the proposition. So it remains to show the case where $c(\phi)$ is an integer greater than $1$. Let $j=\min\{i\in[1,N]\ |\ \{\phi_i^t\} \text{ is positive}\}$. Then thanks to Theorem \ref{translation selector} we have
\begin{equation}\label{mega relou}
\left\lceil c\left(\prod\limits_{i=j+1}^N\phi_i^1\right)\right\rceil \geq \left\lceil c\left(\prod\limits_{i=j}^N\phi_i^1\right)-c(\phi_j^1)\right\rceil\ .
\end{equation}

\noindent
Let us assume by contradiction that $c(\phi)=k$. This implies as for the discriminant norm that for any $i\in[1,N]$ such that $\{\phi_i^t\}$ is positive, we must have $c(\phi_i^1)\in]0,1[$. By the fourth point of Theorem \ref{translation selector} we have $k=c(\phi)\leq c(\prod\limits_{i=j}^N\phi_i^1)$ and by the triangular inequality we have $c(\prod\limits_{i=j}^N\phi_i^1)\leq k$. So we deduce that $c(\prod\limits_{i=j}^N\phi_i^1)=k$. Plugging this in the inequality \eqref{mega relou} we obtain the following contradiction:
\[k=c\left(\prod\limits_{i=j}^N\phi_i^1\right)<\left\lceil c(\phi_j^1)\right\rceil +\left\lceil c\left(\prod\limits_{i=j+1}^n\phi_i^1\right)\right\rceil \leq \sum\limits_{i=j}^N\lceil c(\phi_i^1)\rceil=k.\]\end{proof}

\vspace{2.ex}

\noindent

\subsection{Computation of the discriminant and oscillation lengths and proof of Theorem \ref{geodesique discriminante/oscillation}}

The next lemma will allow us to compute the discriminant and oscillation lengths of paths generated by time independent Hamiltonian functions and prove Theorem \ref{geodesique discriminante/oscillation}.\\

\begin{lemme}\label{longueur autonome}
Let $h:\mathbb{R}^{2n}\times S^1\to\mathbb{R}$ be a smooth compactly supported Hamiltonian function that generates the path of compactly supported contactomorphisms $\phi_h$. Then
\[\mathcal{L}_d\left(\phi_h\right)=\left\lfloor \frac{1}{t_0}\right\rfloor+1\ \text{ where }\ t_0=\inf\left\{t>0\ |\ \Interior(\Supp(h))\cap DP(\phi_h^t)\ne\emptyset \right\} .\]
If moreover we suppose that $h:\mathbb{R}^{2n}\times S^1\to\mathbb{R}$ is non-negative (resp.\ non-positive), then
\[\mathcal{L}_d\left(\phi_h\right)=\mathcal{L}_{osc}\left(\phi_h\right)=\left\lfloor\frac{1}{t_0}\right\rfloor+1,\]
where by convention we set $\frac{1}{0}=+\infty$.\\
\end{lemme}

\begin{proof}
 Recall that \[\Supp(\phi_h)=\overline{\underset{t\in[0,1]}{\bigcup}\{x\in \mathbb{R}^{2n}\times S^1\ |\ \phi_h^t(x)\ne x\}}\ \text{ and }\  \Supp(h)=\overline{\{x\in \mathbb{R}^{2n}\times S^1 \ |\ h(x)\ne 0\}}.\]
 Straightforward arguments allow to show that 
 \[\Supp(\phi_h)=\Supp(h)\ ,\] 
 so $t_0=\inf\left\{t>0\ |\ \Interior(\Supp(\phi_h))\cap DP(\phi_h^t)\ne\emptyset \right\}$.\\
 
 \noindent
 Suppose there exist $0\leq t<s\leq 1 $ such that $DP((\phi_h^t)^{-1}\phi_h^s)\cap\Interior(\Supp(\phi_h))\ne\emptyset$. Since the Hamiltonian function $h$ does not depend on time this implies that $(\phi_h^t)^{-1}\phi_h^s=\phi_h^{s-t}$ and so $DP(\phi_h^{s-t})\cap\Interior(\Supp(\phi_h))\ne\emptyset$. So by definition of $t_0$ we have that $s-t\geq t_0$.\\
 
 \noindent
So if $t_0>0$ and $\frac{1}{t_0}$ is not an integer, by cutting the interval $[0,1]$ in $\left\lceil\frac{1}{t_0}\right\rceil$ intervals of same length $1/\left\lceil\frac{1}{t_0}\right\rceil$, the discriminant length of $\phi_h$ restricted to any of this interval is equal to one, i.e.\ 
\[\mathcal{L}_d\left(\{\phi_h^t\}_{t\in[i/\lceil\frac{1}{t_0}\rceil; (i+1)/\lceil\frac{1}{t_0}\rceil]}\right)=1,\ \forall i\in \left[0,\left\lceil \frac{1}{t_0}\right\rceil-1\right]\ .\]
Therefore $\mathcal{L}_d(\phi_h)\leq \left\lceil \frac{1}{t_0}\right\rceil$. Moreover if we cut $[0,1]$ in stricly less than $\left\lceil \frac{1}{t_0}\right\rceil$ intervals, then there exists at least one interval $I\subset [0,1]$ with length greater than $t_0$ and so $\mathcal{L}_d\left(\{\phi_h^t\}_{t\in I}\right)\geq 2$. We then deduce that when $1>t_0>0$ and $\frac{1}{t_0}$ is not an integer we have 
\[ \mathcal{L}_d(\phi_h)=\left\lceil \frac{1}{t_0}\right\rceil=\left\lfloor\frac{1}{t_0}\right\rfloor+1 \ .\]

\noindent
 Now let us consider the case when $t_0>0$ and $\frac{1}{t_0}$ is an integer. First we can see that if we cut $[0,1]$ into $\frac{1}{t_0}$ intervals then at least one of them will be of length greater or equal to $t_0$, and so the discriminant length of $\phi_h$ restricted to this interval will be greater or equal than $2$. So we deduce that $\mathcal{L}_d(\phi_h)\geq \frac{1}{t_0}+1$. However if we cut $[0,1]$ into $\frac{1}{t_0}+1$ pieces such that each piece is of length $1/(\frac{1}{t_0}+1)$ then the same argument as before allows to show that the discriminant length of $\phi_h$ restricted to any of these intervals is equal to one. So we conclude that in this case again 
 \[\mathcal{L}_d(\phi_h)=\left\lfloor \frac{1}{t_0}\right\rfloor +1.\]

\noindent
 Finally in the case $t_0=0$ it is obvious that $\mathcal{L}_d(\phi_h)=+\infty$. \\
\end{proof}

\noindent
We are now ready to prove Theorem \ref{geodesique discriminante/oscillation}.\\

\begin{proof}[Proof of Theorem \ref{geodesique discriminante/oscillation}]
Let us first compute the $t_0$ coming from Lemma \ref{longueur autonome} :
\[t_0=\inf\left\{t>0\ |\ \Interior\left({\Supp(h)}\right)\cap DP(\phi_h^t)\ne\emptyset\right\}\ .\]
For all $0<t\leq 1$, $(p,z)\in DP(\phi_h^t)$ if and only if $(p,z)$ is a $k$-translated point for $\phi_h^t$, where $k$ is an integer. Using Lemmas \ref{relevé de contact} and \ref{petite hessienne}, we know that $(p,z)$ is a $k$-translated point of $\phi_h^t$, for some $t\in ]0,1]$, if and only if $p$ is a critical point of $H$ and $tH(p)=k$ (see the proof of Lemma \ref{calcul explicite}). So
\[t_0=\inf\left\{t>0 \ |\ \exists (p,z)\in\Interior\left({\Supp(h)}\right),\ d_pH=0\ \text{ and }\ tH(p)\in\mathbb{Z}\right\}.\]
Since $\Supp(h)=\Supp(H)\times S^1$ and since $0$ is a regular value of $H$ inside of its support, we deduce that 
\[\begin{aligned}
t_0&=\min\left\{t>0\ |\ \exists p\in \Interior\left(\Supp(H)\right),\  d_p H=0\ \text{ and } tH(p)\in\mathbb{Z}\setminus\{0\}\right\}\\
&=\min\left\{t>0 \ |\ \exists p\in \Interior\left(\Supp(H)\right),\  d_p H=0\ \text{ and } t|H(p)|=1\right\}\\
&=\min\left\{t>0 \ |\ \exists p\in \Interior\left(\Supp(H)\right),\  d_p H=0\ \text{ and } t=\frac{1}{|H(p)|}\right\}\\
&=\frac{1}{\max\{\max H,-\min H\}}=\frac{1}{\max\{\max h,-\min h\}}.
\end{aligned}\]
So by Lemma \ref{longueur autonome} we have \[\mathcal{L}_d(\phi_h)=\max\{\lfloor \max h \rfloor +1,\lfloor -\min h+\rfloor+1\}.\]
On the other side, Proposition \ref{crucial} and Lemma \ref{calcul explicite} allow us to say that 
\[\nu_d\left(\phi_h^1\right)\geq \max\left\{\lfloor c(\phi_h^1)\rfloor+1, \left\lfloor c\left((\phi_h^1)^{-1}\right)\right\rfloor+1\right\}=\max\left\{\lfloor \max h \rfloor+1, \lfloor \min h\rfloor+1 \right\}.\]
By definition $\mathcal{L}_d(\phi_h)\geq\nu_d(\phi_h^1)$ all these inequalities are equalities :
\[\max\{\lfloor \max H+1 \rfloor ,\lfloor -\min H+1 \rfloor\}=\mathcal{L}_d(\phi_h)\geq\nu_d(\phi_h^1)\geq \max\{\lfloor \max H+1 \rfloor ,\lfloor -\min H+1 \rfloor\}.\]
Finally, because the discriminant length and the discriminant norm are conjugation invariant we deduce Theorem \ref{geodesique discriminante/oscillation}. \\

\noindent
The proof for the oscillation norm goes exactly the same way. \end{proof}

\vspace{2.ex}

\subsection{Proof of Proposition \ref{capacite-energie}}
Let us recall the statement of Proposition \ref{capacite-energie}.

\begin{proposition}
Let $\phi\in\Cont_0^c(\mathbb{R}^{2n}\times S^1,\xi_{st})$ be a contactomorphism that displaces $U\subset\mathbb{R}^{2n}\times S^1$, i.e.\ $\phi(U)\cap U=\emptyset$. Then $\left\lceil \nu(\phi)\right\rceil \geq\frac{1}{2}\left\lceil c(U)\right\rceil $ where $\nu$ denotes the FPR, discriminant, oscillation or Shelukhin norm.\\ 
\end{proposition}

The proof of Proposition \ref{capacite-energie} is an immediate consequence of Corollary \ref{corollaire shelukhin}, Proposition \ref{crucial} together with the inequality \eqref{sandoninegalite} of the introduction.

\begin{proof}[Proof of Proposition \ref{capacite-energie}]
From Corollary \ref{corollaire shelukhin} and Proposition \ref{crucial} we deduce that 
$2\left\lceil \nu(\phi)\right\rceil\geq \left\lceil c(\phi)\right\rceil+\left\lceil c(\phi^{-1})\right\rceil$, where $\nu$ stands for any of the norms $\nu_d,\ \nu_{osc},\ \nu^{\alpha_{st}}_{FPR}$ or $\nu^{\alpha_{st}}_S$, which concludes the proof of Proposition \ref{capacite-energie}.
\end{proof}

\bibliographystyle{plain}
\bibliography{biblio}

\end{document}